\documentclass[smallcondensed]{svjour3}     
\smartqed  
\usepackage[pdftex]{graphicx}
\usepackage{soul}
\usepackage{psfrag}
\usepackage{a4wide}
\usepackage{rotating}
\usepackage{color}
\usepackage{amssymb,amsmath}
\usepackage{verbatim}
\usepackage{tikz}
\newcommand{\ab}[1]{\boldsymbol{#1}}
\def\bfm#1{\boldsymbol{#1}}
\newcommand{\f}[1]{\mathbf{#1}}
\newcommand{\g}{f}
\newcommand{\gC}{f}
\newcommand{\AS}{A}
\newcommand{\bb}[1]{\bfm{#1}}
\newcommand{\N}{\mathbb N}

\newcommand{\R}{\mathbb R}

\newcommand{\sm}{s}
\newcommand{\A}{a}
\newcommand{\B}{b}
\newcommand{\Coeff}{B}
\newcommand{\coe}{d}
\newcommand{\LL}{i_0}
\newcommand{\RR}{i_1}
\newcommand{\Side}{\tau}
\newcommand{\dd}{\partial}
\newcommand{\ff}{f^{\Xi^{(i)}}}

\DeclareMathOperator{\Span}{span}

\begin{document}

\title{$C^{\sm}$-smooth isogeometric spline spaces over planar multi-patch parameterizations
}


\author{Mario Kapl \and Vito Vitrih}


\institute{M. Kapl \at
           Johann Radon Institute for Computational and Applied Mathematics, Austrian Academy of Sciences\\
           Altenbergerstra\ss{}e 69, 4040 Linz, Austria \\
           \email{mario.kapl@ricam.oeaw.ac.at}
         \and
           V. Vitrih \at
           IAM and FAMNIT, University of Primorska\\ 
           Muzejski trg 2, 6000 Koper, Slovenia \\
           \email{vito.vitrih@upr.si}
}

\date{Received: date / Accepted: date}

\maketitle

\begin{abstract}
The design of globally $C^{\sm}$-smooth ($\sm \geq 1$) isogeometric spline spaces over multi-patch geometries is a current and challenging topic 
of research in the framework of isogeometric analysis. In this work, we extend the recent methods~\cite{KaSaTa17a,KaSaTa19a} and 
\cite{KaVi17c,KaVi19a,KaVi20} for the construction of $C^1$-smooth and $C^2$-smooth isogeometric spline spaces over particular planar 
multi-patch geometries to the case of $C^{\sm}$-smooth isogeometric multi-patch spline spaces of an arbitrary selected smoothness 
$\sm \geq 1$. More precisely, for any~$\sm \geq 1$, we study the space of $C^{\sm}$-smooth isogeometric spline functions defined on planar, 
bilinearly parameterized multi-patch domains, and generate a particular $C^{\sm}$-smooth subspace of the entire $C^{\sm}$-smooth 
isogeometric multi-patch spline space. We further present the construction of a basis for this $C^{\sm}$-smooth  subspace, 
which consists of simple and locally supported functions. Moreover, we use the $C^{\sm}$-smooth spline functions to perform $L^2$ 
approximation on bilinearly parameterized multi-patch domains, where the obtained numerical results indicate an optimal approximation 
power of the constructed $C^{\sm}$-smooth subspace.
\end{abstract}

\section{Introduction} \label{sec:introduction}

Multi-patch spline geometries with possibly extraordinary vertices, i.e. vertices with valencies different to four, are a useful tool in 
Computer Aided Design~\cite{Fa97,HoLa93} for modeling complex objects, which usually 
cannot be described just by single-patch geometries. The concept of isogeometric analysis~\cite{ANU:9260759,CottrellBook,HuCoBa04} allows the 
construction of globally $C^{\sm}$-smooth ($\sm \geq 1$) isogeometric spline spaces over these multi-patch geometries. The smooth spline spaces can 
then be used to solve high order partial differential equations (PDEs) on the multi-patch domains directly via the weak form and a 
standard Galerkin discretization. While in case of fourth order PDEs such as the biharmonic equation, 
e.g.~\cite{BaDe15,CoSaTa16,KaBuBeJu16,NgKaPe15,TaDe14}, the Kirchhoff-Love shell problem, e.g. 
\cite{ABBLRS-stream,benson2011large,kiendl-bazilevs-hsu-wuechner-bletzinger-10,kiendl-bletzinger-linhard-09,KiHsWuRe15}, 
the Cahn-Hilliard equation, e.g. \cite{gomez2008isogeometric,GoCaHu09,LiDeEvBoHu13}, and problems of strain gradient elasticity, 
e.g.~\cite{gradientElast2011,MaReBeJu18,KhakaloNiiranenC1}, $C^1$-smooth isogeometric spline functions are needed,  
even $C^2$-smooth functions are required in case of sixth order PDEs such as the triharmonic equation, e.g. 
\cite{BaDe15,KaVi19a,TaDe14}, the phase-field crystal equation, e.g. \cite{BaDe15,Gomez2012}, the Kirchhoff plate model based 
on the Mindlin's gradient elasticity theory, e.g. \cite{KhNi17,Niiranen2016}, and the gradient-enhanced continuum damage model, 
e.g. \cite{GradientDamageModels}. Isogeometric collocation, see 
e.g.~\cite{SuperConvergent2015,IsoCollocMethods2010,GomezLorenzisVariationalCollocation,KaVi20,MonSanTam2017}, is another possible application of globally 
smooth isogeometric spline spaces. The solving of the strong form of the PDE requires now in case of a second order PDE $C^2$-smooth isogeometric spline 
functions and in case of a fourth order PDE already $C^4$-smooth functions.

The construction of globally $C^{\sm}$-smooth isogeometric spline spaces over multi-patch geometries is mainly based on the observation that an 
isogeometric function is $C^{\sm}$-smooth over the given multi-patch domain if and only if its associated multi-patch graph surface is 
$G^{\sm}$-smooth (i.e. geometrically continuous of order~$\sm$). In this work, we will focus on the design of smooth isogeometric spline spaces 
over planar multi-patch geometries. So far, most existing techniques are limited to a global smoothness of~$\sm=1$ and $\sm=2$. In case of 
$\sm=1$, these methods can be roughly classified into three approaches depending on the used multi-patch parameterization. While the first 
strategy employs a multi-patch parameterization, which is $C^1$-smooth everywhere and therefore possesses a singularity at the extraordinary 
vertices, see e.g. \cite{NgPe16,ToSpHu17}, the second approach uses a multi-patch parameterization, which is $C^1$-smooth everywhere except in the 
neighborhood of an extraordinary vertex, where a special construction of the parameterization is needed, see e.g 
\cite{Pe15-2,KaPe17,KaPe18,NgKaPe15}. In contrast to the first two approaches, where the multi-patch geometry is $C^1$-smooth at most parts of 
the multi-patch domain, the used multi-patch parameterization in the third stratey is in general just $C^0$-smooth at the interfaces. Examples 
of such parameterizations are (mapped) piecewise bilinear parameterizations, e.g.~\cite{BeMa14,KaBuBeJu16,KaViJu15}, 
general analysis-suitable parameterizations, e.g.~\cite{CoSaTa16,KaSaTa17a,KaSaTa17b,KaSaTa19a}, non-analysis-suitable parameterizations, 
e.g.~\cite{ChAnRa18,ChAnRa19}, and general quadrilateral meshes of arbitrary topology~\cite{BlMoVi17,BlMoXu20,mourrain2015geometrically}.
The recent survey article~\cite{KaSaTa19b} provides more details about the single methods of the three approaches and also includes further 
possible constructions.

In case of $\sm=2$, there exist only a small number of possible constructions, which mainly follow the third strategy for~$\sm=1$, 
see e.g.~\cite{KaVi17a,KaVi17b,KaVi17c,KaVi19a,KaVi20}. All these methods can be applied to the case of (mapped) bilinear multi-patch 
parameterizations, but the techniques~\cite{KaVi17c,KaVi19a,KaVi20} work also for a more general class of multi-patch parameterizations, 
called bilinear-like $G^2$ multi-patch geometries, cf.~\cite{KaVi17c}. The design of $C^{\sm}$-smooth isogeometric spline spaces for 
planar multi-patch geometries with possibly extraordinary vertices has not been considered so far for a global smoothness of 
$\sm > 2$, and is the topic of this paper. A related approach, which is based on a polar configuration and enables the construction of 
$C^{\sm}$-smooth isogeometric spline functions with a smoothness of~${\sm} \geq 3$, is the technique~\cite{ToSpHiHu16}.

In this paper, we study and generate $C^{\sm}$-smooth isogeometric spline functions of an arbitrary smoothness~$\sm \geq 1$, 
which are defined over planar, multi-patch parameterizations. The construction is mainly described for the case of bilinearly parameterized 
multi-patch domains, but can be enlarged to the wider class of bilinear-like $G^{\sm}$ multi-patch geometries, which has been already introduced 
for the case~$\sm=2$ in \cite{KaVi17c}, and which allows the modeling of planar multi-patch geometries with curved interfaces and boundaries. 
The presented study and construction of the globally $C^{\sm}$-smooth isogeometric spline functions can be seen as an extension of the 
techniques~\cite{KaSaTa17a,KaSaTa19a} and \cite{KaVi17c,KaVi19a,KaVi20} for the design of $C^{\sm}$-smooth isogeometric multi-patch spline spaces 
for the case of $\sm=1$ and $\sm=2$, respectively. More precisely, we develop for the case of a planar bilinear multi-patch parameterization a 
theoretical framework to study the $C^{\sm}$-smoothness condition of an isogeometric function and to characterize the resulting $C^{\sm}$-smooth 
function. We also use this framework to generate a particular $C^{\sm}$-smooth isogeometric spline space for a given planar, 
bilinearly parameterized multi-patch domain and to construct a simple and locally supported basis for the $C^s$-smooth space. 
Several numerical tests by performing $L^2$~approximation using the $C^{\sm}$-smooth isogeometric spline space for different~$\sm$ indicate 
an optimal approximation power of the constructed $C^{\sm}$-smooth space, and demonstrate the potential of the space
for the use in isogeometric analysis.

The remainder of this paper is organized as follows. In Section~\ref{sec:multipatch}, we introduce the particular class of planar multi-patch 
geometries, which consists of bilinearly parameterized quadrilateral patches, and will be used throughout the paper. Moreover, we present the 
concept of $C^{\sm}$-smooth isogeometric spline spaces over this class of multi-patch geometries. Section~\ref{sec:two-patch-case} studies the 
$C^{\sm}$-smoothness condition of an isogeometric function across two neighboring patches and describes first the construction of a particular 
$C^{\sm}$-smooth isogeometric spline space for the case of a bilinearly parameterized two-patch domain. In Section~\ref{sec:design_multipatches}, 
we then extend the particular construction to the case of bilinearly parametrized multi-patch domains with more than two patches and with possibly 
extraordinary vertices. For both cases, we also explain the design of a simple basis, which consists of locally supported functions. A first 
possible generalization of our approach beyond bilinear parameterizations is briefly discussed in Section~\ref{sec:bilinearlike}. Numerical 
experiments in Section~\ref{sec:examples} indicate optimal approximation properties of the presented $C^{\sm}$-smooth isogeometric multi-patch 
spline spaces. Finally, we conclude the paper in Section~\ref{sec:conclusion}.

\section{The multi-patch setting and $C^{\sm}$-smooth isogeometric spline spaces} \label{sec:multipatch}

In this section, we will first describe the multi-patch setting, which will be used throughout the paper, and then, we will give a 
short overview of the concept of $C^{\sm}$-smooth ($\sm \geq 1$) isogeometric spline spaces over the considered class of multi-patch domains.

Let $\Omega$ and $\Omega^{(i)}$, $i \in \mathcal{I}_{\Omega}$, be open and connected domains in $\R^2$, such that 
$\overline{\Omega} = \cup_{i \in \mathcal{I}_{\Omega}} \overline{\Omega^{(i)}}$, where $\mathcal{I}_{\Omega}$ is the index set of the
indices of the patches~$\Omega^{(i)}$. Furthermore, let 
$\Omega^{(i)}$, $i \in \mathcal{I}_{\Omega}$, be quadrangular patches, which are mutually disjoint, and the closures of any two of them have either 
an empty intersection, possess exactly one common vertex or share the whole common edge. Additionally, the deletion of any vertex of the 
multi-patch domain~$\overline{\Omega}$ does not split $\overline{\Omega}$ into subdomains, whose union would be unconnected. We will further 
assume that each patch $\overline{\Omega^{(i)}}$ is parameterized by a bilinear, bijective and regular geometry mapping~$\ab{F}^{(i)}$,
\begin{align*}
 \ab{F}^{(i)}: [0,1]^{2}  \rightarrow \R^{2}, \quad 
 \bb{\xi} =(\xi_1,\xi_2) \mapsto
 \ab{F}^{(i)}(\bb{\xi}) = \ab{F}^{(i)}(\xi_1,\xi_2), \quad i \in \mathcal{I}_{\Omega},
\end{align*}
such that $\overline{\Omega^{(i)}} = \ab{F}^{(i)}([0,1]^{2})$, see Fig.~\ref{fig:geommetryToOmega}.
In addition, we denote by $\ab{F}$ the multi-patch parameterization consisting of all geometry mappings~$\ab{F}^{(i)}$, $i \in 
\mathcal{I}_{\Omega}$. We will also use the splitting of the multi-patch domain $\overline{\Omega}$ into the single patches~$\Omega^{(i)}$, 
$i \in \mathcal{I}_{\Omega}$, edges~$\Gamma^{(i)}$, $i \in \mathcal{I}_{\Gamma}$, and vertices~${\Xi}^{(i)}$, $i \in \mathcal{I}_{\Xi}$, i.e.,
\begin{equation*} 
\displaystyle
\overline{\Omega} = \bigcup_{i \in \mathcal{I}_{\Omega}} \Omega^{(i)}  \; \dot{\cup}  \bigcup_{i \in \mathcal{I}_{\Gamma}} \Gamma^{(i)} \; 
\dot{\cup} \bigcup_{i \in \mathcal{I}_{\Xi}} {\Xi}^{(i)},
\end{equation*}
where $\mathcal{I}_{\Gamma}$ and $\mathcal{I}_{\Xi}$ are the index sets of the indices of the edges~$\Gamma^{(i)}$ and vertices~${\Xi}^{(i)}$, 
respectively.

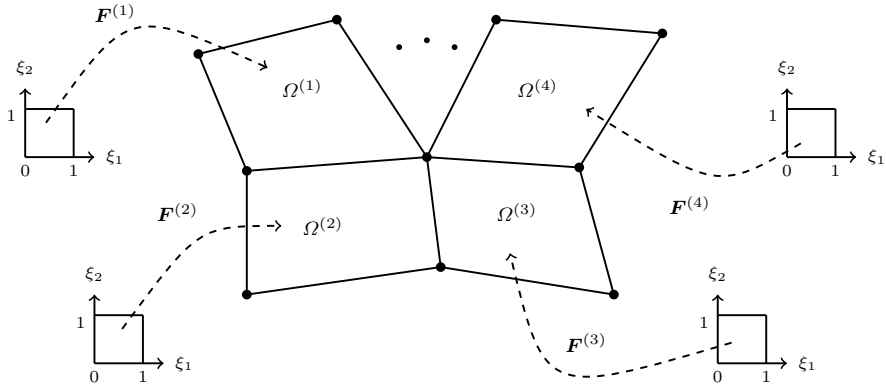
\begin{figure}[hbt]
\begin{center}
\resizebox{0.8\textwidth}{!}{
 \begin{tikzpicture}
  \coordinate(A) at (0.8,0);  \coordinate(B) at (3.,-0.15); \coordinate(C) at (4.2,1.8); \coordinate(D) at (-1.8,-0.2);
  \coordinate(E) at (-1.8,-2); \coordinate(F) at (1.0,-1.6); \coordinate(G) at (3.5,-2.); 
  \coordinate(H) at (1.8,2);
  \coordinate(I) at (-0.5,2);   \coordinate(J) at (-2.5,1.5);
  \coordinate(X) at (0.4,1.6);
  \coordinate(Y) at (1.2,1.6);
  \coordinate(Z) at (0.8,1.7);

  \draw[thick] (A) -- (B); \draw[thick] (A) -- (D);  \draw[thick] (A) -- (F); 
  \draw[thick] (B) -- (C);  \draw[thick] (D) -- (E); \draw[thick] (E) -- (F); \draw[thick] (F) -- (G); \draw[thick] (G) -- (B);
  \draw[thick] (C) -- (H); \draw[thick] (H) -- (A);
  \draw[thick] (D) -- (J);  \draw[thick] (J) -- (I); \draw[thick] (I) -- (A);

  \fill (A) circle (2.pt); \fill (B) circle (2pt); \fill (C) circle (2pt); \fill (D) circle (2pt);
  \fill (E) circle (2.pt); \fill (F) circle (2pt); \fill (G) circle (2pt);  \fill (H) circle (2pt); 
   \fill (I) circle (2.pt);    \fill (J) circle (2pt);
  \fill (X) circle (1.2pt);
  \fill (Y) circle (1.2pt);
  \fill (Z) circle (1.2pt);

  \draw[thick,->] (-4,-3) -- (-4.,-2);
  \draw[thick,->] (-4,-3) -- (-3.,-3);
  \draw[thick] (-4,-2.3) -- (-3.3,-2.3);
  \draw[thick] (-3.3,-3) -- (-3.3,-2.3);
  \node at (-2.7,-3) {\scriptsize $\xi_1$};
  \node at (-4,-1.7) {\scriptsize $\xi_2$};
   \node at (-3.3,-3.2) {\scriptsize $1$};
   \node at (-4.2,-2.4) {\scriptsize $1$};
  \node at (-4,-3.2) {\scriptsize $0$};
  
    \draw[thick,->] (-5,0) -- (-5.,1);
  \draw[thick,->] (-5,0) -- (-4.,0);
  \draw[thick] (-5,0.7) -- (-4.3,0.7);
  \draw[thick] (-4.3,0) -- (-4.3,0.7);
  \node at (-3.7,0) {\scriptsize $\xi_1$};
  \node at (-5,1.3) {\scriptsize $\xi_2$};
   \node at (-4.3,-0.2) {\scriptsize $1$};
   \node at (-5.2,0.6) {\scriptsize $1$};
  \node at (-5,-0.2) {\scriptsize $0$};
  
  \draw[thick,->] (5,-3) -- (5.,-2);
  \draw[thick,->] (5,-3) -- (6.,-3);
  \draw[thick] (5,-2.3) -- (5.7,-2.3);
  \draw[thick] (5.7,-3) -- (5.7,-2.3);
  \node at (6.3,-3) {\scriptsize $\xi_1$};
  \node at (5,-1.7) {\scriptsize $\xi_2$};
   \node at (5.7,-3.2) {\scriptsize $1$};
   \node at (4.8,-2.4) {\scriptsize $1$};
  \node at (5,-3.2) {\scriptsize $0$};
  
    \draw[thick,->] (6,0) -- (6.,1);
  \draw[thick,->] (6,0) -- (7.,0);
  \draw[thick] (6,0.7) -- (6.7,0.7);
  \draw[thick] (6.7,0) -- (6.7,0.7);
  \node at (7.3,0) {\scriptsize $\xi_1$};
  \node at (6,1.3) {\scriptsize $\xi_2$};
   \node at (6.7,-0.2) {\scriptsize $1$};
   \node at (5.8,0.6) {\scriptsize $1$};
  \node at (6,-0.2) {\scriptsize $0$};

  \node at (-1.,1.) {\small $\Omega^{(1)}$};
  \node at (-0.7,-1) {\small $\Omega^{(2)}$};
  \node at (2.1,-0.8) {\small $\Omega^{(3)}$};
  \node at (2.4,1.) {\small $\Omega^{(4)}$};
  
    \node at (-3.7,2.1) {\small $\bfm{F}^{(1)}$};
  \node at (-2.8,-0.8) {\small $\bfm{F}^{(2)}$};
  \node at (3.1,-2.7) {\small $\bfm{F}^{(3)}$};
  \node at (4.6,-0.7) {\small $\bfm{F}^{(4)}$};
  
  \draw[thick,dashed,->] (-4.7,0.5) .. controls(-3.5,2.2) .. (-1.5,1.3);
   \draw[thick,dashed,->] (-3.6,-2.5) .. controls(-2.5,-1) .. (-1.3,-1.);
   \draw[thick,dashed,->] (5.2,-2.7) .. controls(2.5,-3.5) .. (2,-1.4);
   \draw[thick,dashed,->] (6.2,0.2) .. controls(5,-0.5) .. (3.1,0.7);
  \end{tikzpicture}
 }
\end{center}
\caption{The multi-patch domain $\overline{\Omega}= \cup_{i \in \mathcal{I}_{\Omega}}\overline{\Omega^{(i)}}$ with the corresponding geometry 
mappings $\bfm{F}^{(i)}$, $i \in \mathcal{I}_{\Omega}$. } 
\label{fig:geommetryToOmega}
\end{figure}

Let us describe now the isogeometric spline spaces that will be considered in this work. 
We denote by $\mathcal{S}_h^{p,r}([0,1])$ the univariate spline space of degree~$p$, regularity~$r$ and mesh size~$h=\frac{1}{k+1}$, 
which is defined on the unit interval~$[0,1]$, and which is constructed from the uniform open knot vector
\begin{equation*}  
(\underbrace{0,\ldots,0}_{(p+1)-\mbox{\scriptsize times}},
\underbrace{\textstyle \frac{1}{k+1},\ldots,\frac{1}{k+1}}_{(p-r) - \mbox{\scriptsize times}},\ldots, 
\underbrace{\textstyle \frac{k}{k+1},\ldots ,\frac{k}{k+1}}_{(p-r) - \mbox{\scriptsize times}},
\underbrace{1,\ldots,1}_{(p+1)-\mbox{\scriptsize times}}),
\end{equation*}
where $k$ is the 
number of different inner knots. Furthermore, let $\mathcal{S}_h^{\ab{p},\ab{r}}([0,1]^2)$ be the tensor-product spline space 
$\mathcal{S}_h^{p,r}([0,1]) \otimes \mathcal{S}_h^{p,r}([0,1])$ on the unit-square~$[0,1]^2$. We denote the B-splines of the 
spaces~$\mathcal{S}_h^{p,r}([0,1])$ and $\mathcal{S}_h^{\ab{p},\ab{r}}([0,1]^2)$ by $N_{j}^{p,r}$ and $N_{j_1,j_2}^{\ab{p},
\ab{r}}=N_{j_1}^{p,r}N_{j_2}^{p,r}$, respectively, with $j,j_1,j_2=0,1,\ldots,n-1$, where $n=p+1+k(p-r)$.
We assume that $p \geq 2\sm +1$ and $\sm \leq r \leq p-(\sm+1)$. 
Since the geometry mappings~$\ab{F}^{(i)}$, $i \in \mathcal{I}_{\Omega}$, are bilinearly parameterized, we trivially have that
\[
\ab{F}^{(i)} \in \mathcal{S}_{h}^{\ab{p},\ab{r}}([0,1]^2) \times \mathcal{S}^{\ab{p},\ab{r}}_{h}([0,1]^2).
\]
The space of
isogeometric functions on~$\Omega$ is given as
\begin{equation*}
\mathcal{V}^{} = \left\{ \phi \in L^2(\overline{\Omega}) \; | \;  \phi |_{\overline{\Omega^{(i)}}} \in {\mathcal{S}_{h}^{\ab{p},\ab{r}}([0,1]^{2})} 
\circ (\ab{F}^{(i)})^{-1}, \; i \in \mathcal{I}_{\Omega}   \right\}.
\end{equation*}
In addition, let
\begin{equation*}
\mathcal{V}^{\sm} =  \mathcal{V} \cap \mathcal{C}^\sm (\overline{\Omega})
\end{equation*}
be the space of $C^s$-smooth isogeometric functions on $\Omega$. For an isogeometric function~$\phi \in \mathcal{V}$, we denote the spline 
functions $\phi \circ \ab{F}^{(i)}$, $i \in \mathcal{I}_{\Omega}$, by $\g^{(i)}$, and specify their spline representations by
\[
 \g^{(i)}(\xi_1,\xi_2) = \sum_{j_1=0}^{n-1}\sum_{j_2=0}^{n-1} \coe^{(i)}_{j_1,j_2} N_{j_1,j_2}^{\ab{p},\ab{r}}(\xi_1,\xi_2), \quad \coe^{(i)}_{j_1,j_2} \in \R.
\]
Moreover, we define the graph~$\bfm{\Sigma} \subseteq \Omega \times \R$ of an isogeometric function~$\phi \in \mathcal{V}$ as the collection of 
the graph surface patches~$\bfm{\Sigma}^{(i)}:[0,1]^2 \rightarrow \Omega^{(i)} \times \R$, $i \in \mathcal{I}_{\Omega}$, given by 
\[
 \bfm{\Sigma}^{(i)}(\xi_1,\xi_2) = \left( \ab{F}^{(i)}(\xi_1,\xi_2), \g^{(i)}(\xi_1,\xi_2) \right)^T.
\]
The space $\mathcal{V}^{\sm}$ can be characterized by means of the concept of geometric continuity of multi-patch surfaces, cf.~\cite{HoLa93,Pe02}. 
An isogeometric function~$\phi \in \mathcal{V}$ belongs to the space~$\mathcal{V}^\sm$ if and only if for any two neighboring patches 
$\Omega^{(i_0)}$ and $\Omega^{(i_1)}$, $i_0,i_1 \in \mathcal{I}_{\Omega}$, with the common edge $\overline{\Gamma^{(i)}}= 
\overline{\Omega^{(i_0)}} \cap \overline{\Omega^{(i_1)}}, i \in \mathcal{I}_{\Gamma}$,  
the associated graph surface patches
$ \bfm{\Sigma}^{(i_0)}$ 
and 
$\bfm{\Sigma}^{(i_1)} $
are $G^{\sm}$-smooth, 
see e.g.~\cite{Pe15,KaViJu15}, i.e., there exists a regular, orientation-preserving reparameterization 
$\Phi^{(i)} = \left( \Phi^{(i)}_1,\Phi^{(i)}_2\right)$, $\Phi^{(i)}_j : [0,1]^2 \to 
[0,1]$, $j=1,2$, such that 
\begin{equation}  \label{Gsreparameterization}
 \partial_1^{j_1} \partial_2^{j_2} \, \bfm{\Sigma}^{(i_1)} \big|_{\overline{\Gamma^{(i)}}}=  \partial_1^{j_1} \partial_2^{j_2} 
 \left( \bfm{\Sigma}^{(i_0)} \circ \Phi^{(i)} \right) \big|_{\overline{\Gamma^{(i)}}} \;, \quad 0\leq j_1+j_2 \leq \sm, \; i_0,i_1 \in 
 \mathcal{I}_{\Omega}.
\end{equation}
Here and throughout the paper, we will denote by $\partial_\ell^j$ the $j$-th partial derivative with respect to the $\ell$-th argument of a 
multivariate function, while we will denote by $\partial^j$ the $j$-th derivative with respect to the argument of a univariate function. 

In the next section, first, the case of a two-patch domain will be analyzed.

\section{$C^{\sm}$-smooth isogeometric spline spaces over two-patch domains}  \label{sec:two-patch-case}

In this section, we will restrict ourselves to the case of bilinearly parameterized two-patch domains~$\Omega$. In order to simplify the 
notation, we will denote the patches of the two-patch domain as $\overline{\Omega} = \overline{\Omega^{(\LL)}} \cup 
\overline{\Omega^{(\RR)}}$, their intersection by $\overline{\Gamma} = \overline{\Omega^{(\LL)}} \cap \overline{\Omega^{(\RR)}}$, and the 
corresponding reparameterization just as~$\Phi$. We will first study the $G^s$-smoothness condition of the graph surface of a $C^s$-smooth 
isogeometric spline function defined on a bilinear two-patch domain, and will then use it to construct a particular $C^s$-smooth isogeometric 
spline space. The presented work in this section can be seen as an extension of \cite{KaSaTa17a} and \cite{KaVi17c} for $\sm=1$ and $\sm=2$, 
respectively, to an arbitrary smoothness~$\sm$ in case of bilinear two-patch parameterizations. A possible strategy beyond bilinear 
parameterizations is briefly explained in Section~\ref{sec:bilinearlike}.

\subsection{$G^{\sm}$-smoothness of graph surfaces}

Let $\phi \in \mathcal{V}$, and let $\g^{(\Side)}=\phi \circ \ab{F}^{(\Side)}$, $\Side \in \{\LL,\RR\}$, be the two associated spline 
functions.
To ensure that the 
graph surface $\ab{\Sigma}$
of the isogeometric function~$\phi$, which consists of the two graph surface patches $\ab{\Sigma}^{(\LL)}$ and $\ab{\Sigma}^{(\RR)}$,
is $G^{\sm}$-smooth, the two 
surface patches $\ab{\Sigma}^{(\LL)}$ and $\ab{\Sigma}^{(\RR)}$ have to be joint above their common edge $\Gamma$ with $G^{\sm}$-continuity. 
Without loss of generality, we can assume that $\Phi(0,\xi_2) = (0,\xi_2)$, i.e., the $G^0$ smoothness across the common interface can be written as
\begin{equation} \label{eq:detG0}
 \ab{\Sigma}^{(\LL)}(0,\xi_2) = \ab{\Sigma}^{(\RR)}(0,\xi_2) , \quad \xi_2 \in [0,1].
\end{equation}
In this way the patches $\Omega^{(\LL)}$ and $\Omega^{(\RR)}$ are parameterized as shown in Fig.~\ref{fig:twopatchCase}.
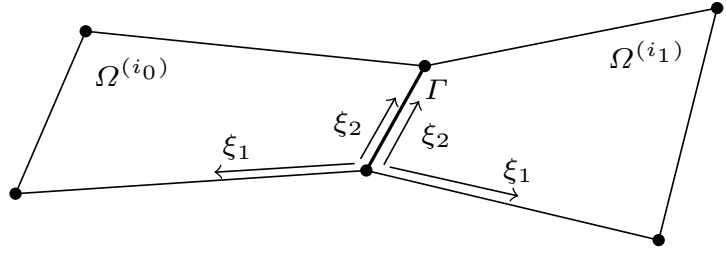
\begin{figure}[hb] 
\begin{center}
\resizebox{0.65\textwidth}{!}{
 \begin{tikzpicture}
  \coordinate(K) at (-0.5,0); \coordinate(\LL) at (0.,0.9); 
  \draw[thick] (\LL) -- (K); 
  \coordinate(A) at (2,-0.6); \coordinate(B) at (2.5,1.4); 
  \draw (A) -- (K); \draw (A) -- (B);  \draw (B) -- (\LL); 

  \coordinate(E) at (-3.5,-0.2); \coordinate(F) at (-2.9,1.2); \coordinate(G) at (-2,3.8); 
  \draw (E) -- (K); \draw (E) -- (F);  \draw (F) -- (\LL); 

  \fill (K) circle (1.5pt); \fill (\LL) circle (1.5pt); 
  \fill (A) circle (1.5pt); \fill (B) circle (1.5pt);  \fill (E) circle (1.5pt); \fill (F) circle (1.5pt);

  \draw[->] (-0.35,0.05) -- (-0.05,0.6);
  \draw[->] (-0.3,0.03) -- (0.8,-0.22);
  \draw[->] (-0.6,0.06) -- (-1.8,-0.015);
  \draw[->] (-0.55,0.1) -- (-0.25,0.63);
 
  \node at (-0.65,0.4) {\scriptsize $\xi_2$};
  \node at (-1.6,0.2) {\scriptsize $\xi_1$};
  \node at (0.8,-0.) {\scriptsize $\xi_1$};
  \node at (0.1,0.3) {\scriptsize $\xi_2$};
   \node at (0.1,0.7) {\scriptsize $\Gamma$};
  \node at (-2.5,0.85) {\scriptsize $\Omega^{(\LL)}$};
  \node at (1.9,1.0) {\scriptsize$\Omega^{(\RR)}$};

  \end{tikzpicture}
 }
 \end{center}
 \caption{The parameterization of the two-patch domain $\Omega^{(\LL)} \cup \Omega^{(\RR)}$ with the common edge $\Gamma$.}
  \label{fig:twopatchCase}
\end{figure}
Furthermore, the $G^1$-smoothness can be expressed as
\begin{equation} \label{eq:detG1}
\det \left( \partial_1 \ab{\Sigma}^{(\RR)}(0,\xi_2) , \partial_1\ab{\Sigma}^{(\LL)}(0,\xi_2) , \partial_2 \ab{\Sigma}^{(\LL)}(0,\xi_2)   
\right) = 0 , \quad \xi_2 \in [0,1],
\end{equation}
which is equivalent to 
\begin{equation} \label{eq:funcG1}
\alpha^{(\LL)}(\xi_2) \partial_1\g^{(\RR)}(0,\xi_2) - \alpha^{(\RR)}(\xi_2) \partial_1 \g^{(\LL)}(0,\xi_2)  - \beta(\xi_2) \partial_2 
\g^{(\LL)}(0,\xi_2) = 0, \quad \xi_2 \in [0,1],
\end{equation}
with
\begin{equation}  \label{eq:alphaLRbar}
 \alpha^{(\Side)}(\xi) = \lambda_1 \det \left( \partial_1 \ab{F}^{(\Side)}(0,\xi), \partial_2 \ab{F}^{(\Side)}(0,\xi) \right), \quad \Side 
 \in \{\LL,\RR \},
\end{equation}
and
$$
\beta(\xi) = \lambda_1 \det \left( \partial_1 \ab{F}^{(\LL)}(0,\xi) , \partial_1\ab{F}^{(\RR)}(0,\xi) \right).
$$ 
We can select $\lambda_1 \in \R$ in such a way, that
\begin{equation} \label{eq:minimizing_alpha}
 || \alpha^{(\LL)}+1 ||^2_{L^2} + || \alpha^{(\RR)}-1 ||^2_{L^2}
\end{equation}
is minimized, cf. \cite{KaVi19a}. Note  
that $\alpha^{(\LL)} < 0$ and $\alpha^{(\RR)} > 0$, since the geometry mappings~$\ab{F}^{(\LL)}$ and 
$\ab{F}^{(\RR)}$ are regular. 
In addition, we can write $\beta$ as
\begin{equation} \label{eq:splitting_beta}
 \beta(\xi) = \alpha^{(\LL)}(\xi) \beta^{(\RR)}(\xi) - \alpha^{(\RR)}(\xi) \beta^{(\LL)}(\xi)
\end{equation}
with
$$
 \beta^{(\Side)}(\xi) = \frac{\partial_1 \ab{F}^{(\Side)}(0,\xi) \cdot \partial_2\ab{F}^{(\Side)}(0,\xi)}
 {||\partial_2 \ab{F}^{(\Side)}(0,\xi)||^{2}}, \quad \Side \in \{\LL,\RR \},
$$
where $\alpha^{(\LL)}, \alpha^{(\RR)}, \beta^{(\LL)}$ and $\beta^{(\RR)}$ are linear polynomials, and $\beta$ is a quadratic one, 
cf. \cite{CoSaTa16,KaVi17c}. 

Recall \eqref{Gsreparameterization}, and let
\begin{equation} \label{eq:aijbij}
  \A_{i,j} (\xi) =  (\partial_1^i \partial_2^j \Phi_1)(0,\xi) , \quad \B_{i,j} (\xi) = (\partial_1^i \partial_2^j \Phi_2)(0,\xi).
\end{equation}
By \eqref{Gsreparameterization}, \eqref{eq:funcG1} and \eqref{eq:aijbij}, we observe that 
\begin{equation} \label{eq:a10b10}
 \A_{1,0}(\xi)  = \frac{\alpha^{(\RR)}(\xi) }{\alpha^{(\LL)}(\xi) }, \quad \B_{1,0}(\xi)  = \frac{\beta(\xi) }{\alpha^{(\LL)}(\xi) }.
\end{equation}

In a similar way as for the $G^1$ smoothness, we can derive conditions for the $G^{\ell}$-smoothness, $2 \leq \ell \leq \sm$ 
(see e.g.~\cite{HoLa93}). For each particular $\ell$ it is enough to consider only the equation
\begin{equation} \label{eq:Gs-condition_graph}
\partial_1^\ell  \ab{\Sigma}^{(\RR)}(0,\xi_2) = \partial_1^\ell \left(\ab{\Sigma}^{(\LL)}\circ 
\Phi \right)(0,\xi_2), \quad \xi_2 \in [0,1],
\end{equation}
since relation \eqref{Gsreparameterization} follows for all other derivatives of order $\ell$ directly from the $G^{\ell-1}$ smoothness. 
Similar to
\eqref{eq:detG1}, we get 
\begin{equation} \label{eq:detGs}
 \det \left( \ab{\Xi}_{\ell} (\xi_2) , \partial_1 \ab{\Sigma}^{(\LL)}(0,\xi_2) , \partial_2 \ab{\Sigma}^{(\LL)}(0,\xi_2)   \right) = 0 , 
 \quad \xi_2 \in [0,1],
\end{equation}
where $\ab{\Xi}_{\ell}$ is given by
\begin{equation}  \label{eq:Xis2}
 \ab{\Xi}_\ell (\xi) = \left(\widetilde{\ab{\Xi}}_{\ell} (\xi),\omega_{\ell} (\xi)  \right) =
 \partial_1^\ell  \ab{\Sigma}^{(\RR)}(0,\xi) - \partial_1^\ell \left(\ab{\Sigma}^{(\LL)}\circ \Phi \right)(0,\xi) + 
 \A_{\ell,0} \, \partial_1 \ab{\Sigma}^{(\LL)} (0,\xi)+  \B_{\ell,0} \, \partial_2 \ab{\Sigma}^{(\LL)} (0,\xi). 
 \end{equation}
Expanding \eqref{eq:detGs} 
leads to
\begin{equation} \label{eq:Gs2}
\lambda_\ell \, \alpha^{(\LL)}(\xi_2) \, \omega_\ell (\xi_2) + \eta_\ell (\xi_2) \, \partial_1 \g^{(\LL)}(0,\xi_2) + \theta_\ell (\xi_2)\, 
\partial_2 \g^{(\LL)}(0,\xi_2) = 0, \quad \xi_2 \in [0,1],
\end{equation}
with
\begin{equation*} \label{eq:Gs2general}
\eta_\ell (\xi) = \lambda_\ell \det \left(  \partial_2 \ab{F}^{(\LL)}(0,\xi),  \widetilde{\ab{\Xi}}_{\ell} (\xi) \right) 
\quad
{\rm and}
\quad
\theta_\ell (\xi) = \lambda_\ell \det \left(  \widetilde{\ab{\Xi}}_{\ell} (\xi), \partial_1 \ab{F}^{(\LL)}(0,\xi) \right). 
\end{equation*}
\begin{remark}
For the sake of simplicity, we will choose $\lambda_\ell=1, \; \ell=2,3,\ldots,\sm$.
\end{remark}
Let 
\begin{equation} \label{eq:cl}
c_\ell (\xi) := \sum_{i=1}^{\ell-1} {\ell \choose i} \A_{i,0}(\xi) \, \B_{\ell-i,0}(\xi), \quad \ell \geq 2,
\end{equation}
then we get in case of the bilinear mapping $\ab{F}^{(\LL)}$
\begin{align}  
\eta_\ell (\xi) &= -  c_\ell (\xi) \, \det \left(  \partial_2 \ab{F}^{(\LL)}(0,\xi),   \partial_1 \partial_2 \ab{F}^{(\LL)}(0,\xi) \right)  = 
\left(\alpha^{(\LL)}(\xi) \right)' \, c_\ell (\xi) , \label{eq:eta}
\end{align}
and
\begin{align}
\theta_\ell (\xi) &= - c_\ell (\xi)  \,\det \left(  \partial_1 \partial_2 \ab{F}^{(\LL)}(0,\xi),  \partial_1 \ab{F}^{(\LL)}(0,\xi) \right) 
\nonumber \\ 
& = \left( \alpha^{(\LL)}(\xi) \left(\beta^{(\LL)}(\xi) \right)' - \left(\alpha^{(\LL)}(\xi) \right)' \beta^{(\LL)}(\xi)\right)   c_\ell (\xi). 
\label{eq:theta}
\end{align}
It further follows from \eqref{eq:detGs}, \eqref{eq:Xis2} and \eqref{eq:Gs2} that
\begin{equation} \label{eq:etatheta}
\eta_\ell(\xi) = - \alpha^{(\LL)}(\xi) \A_{\ell,0}(\xi)\quad {\rm and} \quad 
\theta_\ell(\xi) = - \alpha^{(\LL)}(\xi) \B_{\ell,0}(\xi), \quad 1 \leq \ell \leq \sm,
\end{equation}  
which directly leads to the following theorem.
\begin{theorem} \label{thm1}
The functions $\A_{\ell,0}$ and $\B_{\ell,0}$, $1\leq \ell \leq \sm$, can be expressed by $\alpha^{(\Side)}, \beta^{(\Side)}$, 
$\Side\in \{\LL,\RR\}$, via the recursion 
\begin{align*}
\A_{1,0}(\xi) &= \frac{\alpha^{(\RR)}(\xi)}{\alpha^{(\LL)}(\xi)}, \quad \B_{1,0}(\xi) = \frac{\beta(\xi)}{\alpha^{(\LL)}(\xi)}, \nonumber\\
\A_{\ell,0}(\xi)  & = \vartheta(\xi) c_{\ell}(\xi), \quad \B_{\ell,0}(\xi)= \mu(\xi)c_{\ell}(\xi), \quad 2\leq \ell \leq \sm, 
\label{eq:abell0} 
\end{align*}
where the function $c_{\ell}$ is given in \eqref{eq:cl}, and the functions~$\vartheta$ and $\mu$ are defined as 
\[
 \vartheta(\xi) = - \frac{(\alpha^{(\LL)}(\xi))'}{\alpha^{(\LL)}(\xi)}, \quad 
  \mu(\xi)= - \frac{\left( \alpha^{(\LL)}(\xi) \left(\beta^{(\LL)}(\xi) \right)' - \left(\alpha^{(\LL)}(\xi) \right)' 
\beta^{(\LL)}(\xi)\right)}{\alpha^{(\LL)}(\xi)}.
\]
\end{theorem}

Theorem~\ref{thm1} provides us also closed form formulae for the functions $ \A_{\ell,0}(\xi)$ and $ \B_{\ell,0}(\xi)$, $\ell \geq 2$, which only depend on 
$\alpha^{(i_0)}(\xi), \alpha^{(i_1)}(\xi), \beta^{(i_0)}(\xi)$ and $\beta^{(i_1)}(\xi)$, and which are equal to 
\begin{align*}
 \A_{\ell,0}(\xi) &= \ell ! \, \frac{\alpha^{(i_1)}(\xi) \, \beta_{}(\xi)}{\left(\alpha^{(i_0)}(\xi)\right)^{\ell}} \, \vartheta(\xi) \sum_{j=0}^{\ell-2} N(\ell-1,j+1) \left(\mu(\xi) 
 \,\alpha^{(i_1)}(\xi)\right)^j 
\left(\vartheta(\xi)\, \beta(\xi) \right)^{\ell-2-j}, \\
 \B_{\ell,0}(\xi) &= \ell ! \, \frac{\alpha^{(i_1)}(\xi) \, \beta_{}(\xi)}{\left(\alpha^{(i_0)}(\xi)\right)^{\ell}} \, \mu(\xi) \sum_{j=0}^{\ell-2} N(\ell-1,j+1)  \left(\mu(\xi) \,\alpha^{(i_1)}(\xi)\right)^j 
\left(\vartheta(\xi)\, \beta(\xi) \right)^{\ell-2-j},
\end{align*}
where $\beta(\xi)$ is given in \eqref{eq:splitting_beta} and
$$
N(m_1,m_2) = \frac{1}{m_1} {m_1 \choose m_2} {m_1 \choose m_2-1}, \quad 1\leq m_2 \leq m_1,\;  m_1, m_2 \in \N,
$$ 
are the well-known Narayana numbers. 

\begin{example}
Let us express the functions $\ab{\Xi}_\ell, \eta_\ell$ and $\theta_\ell$ for $\ell \in \{ 1, 2,3\}$. 
Note that the case $\ell=1$ has been already explained above, and that the case $\ell=2$ can be found e.g. in~\cite{KaVi17c}, too.
Using \eqref{eq:Xis2}, \eqref{eq:cl}, \eqref{eq:eta}, \eqref{eq:theta} and \eqref{eq:etatheta}, we get for $\ell=1$
  \begin{equation*}
   \ab{\Xi}_1(\xi) = \partial_1  \ab{\Sigma}^{(\RR)}(0,\xi), \quad
      \eta_1(\xi) = -  \alpha^{(\RR)}(\xi) , \quad  \theta_1(\xi) =  - \beta(\xi),
    \end{equation*}
    for $\ell=2$
 \begin{align*}
   \ab{\Xi}_2(\xi) &= \partial_1^2  \ab{\Sigma}^{(\RR)}(0,\xi) - 
   \left( \A_{1,0}^2(\xi) \partial_1^2  \ab{\Sigma}^{(\LL)}(0,\xi)  +   2 \A_{1,0}(\xi) \B_{1,0}(\xi) \partial_1 \partial_2  
   \ab{\Sigma}^{(\LL)}(0,\xi)   +  \B_{1,0}^2(\xi) \partial_2^2  \ab{\Sigma}^{(\LL)}(0,\xi)  \right), \\
      \eta_2(\xi) &= - \frac{2 \alpha^{(\RR)}(\xi) \beta(\xi) }{(\alpha^{(\LL)}(\xi))^{} } \vartheta(\xi),   \quad
   \theta_2(\xi) = - \frac{2  \alpha^{(\RR)}(\xi) \beta(\xi)} {(\alpha^{(\LL)}(\xi))^{} }  \mu(\xi),
    \end{align*}
    and for $\ell=3$
     \begin{align*} 
     \ab{\Xi}_3(\xi) &= \partial_1^3  \ab{\Sigma}^{(\RR)}(0,\xi) - 
   \left( \A_{1,0}^3(\xi) \partial_1^3  \ab{\Sigma}^{(\LL)}(0,\xi)  +   3 \A_{1,0}^2 (\xi) \B_{1,0}(\xi) \partial_1^2 \partial_2  
   \ab{\Sigma}^{(\LL)}(0,\xi)    \right. \\
   & + \left. 3 \A_{1,0} (\xi) \B_{1,0}^2(\xi) \partial_1 \partial_2^2  \ab{\Sigma}^{(\LL)}(0,\xi)+  \B_{1,0}^3(\xi) \partial_2^3  
   \ab{\Sigma}^{(\LL)}(0,\xi)  + 
   3 \A_{1,0} (\xi) \A_{2,0}(\xi) \partial_1^2 \ab{\Sigma}^{(\LL)}(0,\xi) \right. \\
   & + \left. 3 (\B_{1,0} (\xi) \A_{2,0}(\xi) +\A_{1,0} (\xi) \B_{2,0}(\xi) ) \partial_1 \partial_2  \ab{\Sigma}^{(\LL)}(0,\xi) + 3 \B_{1,0} (\xi) 
   \B_{2,0}(\xi) \partial_2^2 \ab{\Sigma}^{(\LL)}(0,\xi)  
   \right),\\
   \eta_3(\xi) &  =  \frac{6  \alpha^{(\RR)}(\xi) \beta(\xi) 
   \left(\alpha^{(\LL)}(\xi) \alpha^{(\RR)}(\xi)  (\beta^{(\LL)} (\xi) )' + 
   (\alpha^{(\LL)}(\xi))' (\beta(\xi) - \alpha^{(\RR)}(\xi) \beta^{(\LL)}(\xi) )  \right) }
   {(\alpha^{(\LL)}(\xi))^{3} } \vartheta(\xi), \\ 
   \theta_3(\xi) & 
   = \frac{ 6 \alpha^{(\RR)}(\xi) \beta(\xi) 
   \left(\alpha^{(\LL)}(\xi) \alpha^{(\RR)}(\xi)  (\beta^{(\LL)}(\xi) )' + (\alpha^{(\LL)}(\xi))' (\beta(\xi) - \alpha^{(\RR)}(\xi) \beta^{(\LL)}(\xi) )  
   \right) }
   {(\alpha^{(\LL)}(\xi))^{3} } \mu(\xi). \\ 
 \end{align*}
\end{example}

In order to write $\bfm{\Xi}_\ell(\xi)$ explicitly also 
for $\ell > 3$, we need the closed-form expression of 
$$ 
\partial_1^\ell \left(\ab{\Sigma}^{(\LL)}\circ \Phi \right)(0,\xi).
$$ 
This requires the use of the generalized Fa\`{a} di Bruno's formula~\cite{FaaBruno}, i.e.
\begin{equation}  \label{eq:FaaBruno}
   \partial_1^\ell \left(\ab{\Sigma}^{(\LL)}\circ \Phi \right)(0,\xi) = 
   \sum_{|\bfm{\sigma}| = 1}^{\ell} \AS_{\bfm{\sigma};\ell}(\xi) \, \partial_1^{\sigma_1} \partial_2^{\sigma_2} \ab{\Sigma}^{(\LL)}(0,\xi) ,
\end{equation}
where $\bfm{\sigma} = (\sigma_1,\sigma_2)$ is a multi-index with the indices~$\sigma_1,\sigma_2 \in \{0,1,\ldots,\ell \}$, $1 \leq \sigma_1+ \sigma_2 
\leq \ell$, $|\bfm{\sigma}| = \sigma_1+\sigma_2$ is the length of the multi-index, and
\begin{equation*}
 \AS_{\bfm{\sigma};\ell}(\xi) = \ell\, ! \sum_{(\bfm{i},\bfm{j}) \in \mathcal{I}_{\bfm{\sigma};\ell}} \,
 \prod_{\rho=1}^{\ell}  \A_{\rho,0}^{i_\rho} (\xi) \,  \B_{\rho,0}^{j_\rho}(\xi) \frac{1}{\rho ! \, ^{i_\rho+j_\rho} \, i_\rho ! \, j_\rho !},
\end{equation*}
with $\A_{\rho,0}(\xi)$ and  $\B_{\rho,0}(\xi)$ given in \eqref{eq:aijbij}, and
$$
   \mathcal{I}_{\bfm{\sigma};\ell} = \left\{    (\bfm{i},\bfm{j}) = ((i_1,i_2,\ldots,i_\ell),(j_1,j_2,\ldots,j_\ell))\, \Big| \; |\bfm{i}| = 
   \sigma_1,   |\bfm{j}| = \sigma_2, \sum_{\rho=1}^{\ell} \rho (i_\rho + j_\rho) = \ell  \right\}.
$$
\begin{example}
The sets $ \mathcal{I}_{\bfm{\sigma};3}$ as well as the functions~$\AS_{\bfm{\sigma};3}(\xi)$ are equal to
\begin{align*}
 \mathcal{I}_{(1,0);3} &= \{ \left( (0,0,1),(0,0,0) \right) \}, \quad
   \mathcal{I}_{(0,1);3} = \{ \left( (0,0,0),(0,0,1) \right) \}, \\
      \mathcal{I}_{(2,0);3} & = \{ \left( (1,1,0),(0,0,0) \right) \}, \quad  \mathcal{I}_{(0,2);3} = \{ \left( (0,0,0),(1,1,0) \right) \},\\
   \mathcal{I}_{(1,1);3} &= \{ \left( (1,0,0),(0,1,0) \right), \left( (0,1,0),(1,0,0) \right) \}, \\
   \mathcal{I}_{(3,0);3} &= \{ \left( (3,0,0),(0,0,0) \right) \}, \quad
   \mathcal{I}_{(2,1);3} = \{ \left( (2,0,0),(1,0,0) \right) \},  \\
   \mathcal{I}_{(1,2);3} &= \{ \left( (1,0,0),(2,0,0) \right) \}, \quad
   \mathcal{I}_{(0,3);3} = \{ \left( (0,0,0),(3,0,0) \right) \}, 
\end{align*}
and
\begin{align*}
 \AS_{(1,0);3}(\xi) & = \A_{3,0}(\xi),  \quad
   \AS_{(0,1);3}(\xi) = \B_{3,0}(\xi),  \quad
    \AS_{(2,0);3}(\xi) = 3 \A_{1,0}(\xi) \A_{2,0}(\xi),  \\
 \AS_{(1,1);3}(\xi) & = 3 \left( \A_{1,0}(\xi) \B_{2,0}(\xi) +  \B_{1,0}(\xi) \A_{2,0}(\xi) \right),  \quad
   \AS_{(0,2);3}(\xi) = 3 \B_{1,0}(\xi) \B_{2,0}(\xi),  \\
   \AS_{(3,0);3}(\xi) &= \A_{1,0}^3(\xi),  \quad \AS_{(2,1);3}(\xi) = 3 \A_{1,0}^2(\xi) \B_{1,0}(\xi),  \\
   \AS_{(1,2);3}(\xi) &= 3 \A_{1,0}(\xi) \B_{1,0}^2(\xi),  \quad \AS_{(0,3);3}(\xi) = \B_{1,0}^3(\xi).
\end{align*}
\end{example}
\vskip0.5em

Directly following from equations~\eqref{eq:detG0} and \eqref{eq:Gs-condition_graph}, and further summarizing the results of this section, 
we obtain the following proposition.
\begin{proposition} \label{prop:Cs}
Let $\Omega$ be a bilinearly parameterized two-patch domain, i.e. $\overline{\Omega} = \overline{\Omega^{(\LL)}} \cup 
\overline{\Omega^{(\RR)}}$. An isogeometric function 
$\phi \in \mathcal{V}$ belongs to the space $\mathcal{V}^{\sm}$ if and only if 
the two associated spline functions $\g^{(\Side)}=\phi \circ \ab{F}^{(\Side)}$, $\Side \in \{\LL,\RR\}$, fulfill for $\ell =0,1,\ldots,s$,
\begin{equation*} \label{eq:Cs-condition_phi}
\partial_1^\ell  \g^{(\RR)}(0,\xi_2) = \partial_1^\ell \left(\g^{(\LL)}\circ \Phi \right)(0,\xi_2), \quad \xi_2 \in [0,1],
\end{equation*}
 or equivalently
\begin{equation*}
 \g^{(\RR)}(0,\xi_2) = \g^{(\LL)}(0,\xi_2), \quad \xi_2 \in [0,1],
\end{equation*}
and for $\ell=1,2,\ldots,s$,
\begin{equation*}
\alpha^{(\LL)}(\xi_2) \, \omega_\ell (\xi_2) + \eta_\ell (\xi_2) \, \partial_1 \g^{(\LL)}(0,\xi_2) + \theta_\ell (\xi_2)\, 
\partial_2 \g^{(\LL)}(0,\xi_2) = 0, \quad \xi_2 \in [0,1],
\end{equation*}
where $\alpha^{(\LL)} $ is defined via~\eqref{eq:alphaLRbar} and \eqref{eq:minimizing_alpha}, and $\omega_\ell$, $\eta_{\ell}$ and $\theta_{\ell}$ 
are expressed by means of \eqref{eq:Xis2}, \eqref{eq:etatheta} and Theorem~\ref{thm1}. 
\end{proposition}

\subsection{Construction of $C^{\sm}$-smooth isogeometric spline spaces} \label{subsec:design_twopatches}

Proposition~\ref{prop:Cs} describes the $C^s$-smoothness condition for an isogeometric function~$\phi \in \mathcal{V}$.
The following theorem provides now an equivalent but simplified condition, which will be the key step 
for the construction of $C^s$-smooth isogeometric functions.  
\begin{theorem} \label{thm:g_ell}
 Let $\Omega$ be a bilinearly parameterized two-patch domain, i.e. $\overline{\Omega} = \overline{\Omega^{(\LL)}} \cup \overline{\Omega^{(\RR)}}$.
  An isogeometric function 
  $\phi \in \mathcal{V}$
  belongs to the space $\mathcal{V}^{\sm}$ if and only if 
  the corresponding spline functions $\g^{(\LL)}=\phi \circ \ab{F}^{(\LL)}$ and $\g^{(\RR)}=\phi \circ \ab{F}^{(\RR)}$ satisfy
 \begin{equation}   \label{eq:gC}
 \g_\ell^{(\LL)}(\xi) = \g_\ell^{(\RR)}(\xi) =: \gC_\ell(\xi), \quad \ell =0,1,\ldots,\sm,
 \end{equation}
 with
 \begin{equation}   \label{eq:gC2}
 \g_\ell^{(\Side)}(\xi) = \left(\alpha^{(\Side)}(\xi)\right)^{-\ell}\, \partial_1^\ell \g^{(\Side)}(0,\xi) - \sum_{i=0}^{\ell-1} {\ell \choose i} 
 \left(\frac{\beta^{(\Side)}(\xi)}{\alpha^{(\Side)}(\xi)}\right)^{\ell-i}  \dd^{\ell-i} \gC_i^{}(\xi) ,\quad \Side\in \{\LL,\RR\}.
 \end{equation}
\end{theorem}
\begin{proof}
  It directly follows from Proposition~\ref{prop:Cs} that an isogeometric function $\phi \in \mathcal{V}$ belongs to the space 
  $\mathcal{V}^{\sm}$ if and only if the 
  associated spline functions $\g^{(\LL)}$ and 
  $\g^{(\RR)}$ fulfill the equation
  \begin{equation} \label{eq:Cs_modified}
 \left(\alpha^{(\RR)}(\xi)\right)^{-\ell}  \, \left( \partial_1^\ell \g^{(\RR)}(0,\xi)  -  \partial_1^\ell \left(\g^{(\LL)}\circ 
 \Phi \right)(0,\xi)\right) = 0  , \quad \xi \in [0,1],
  \end{equation}
  for $\ell = 0,1,\ldots,\sm$. We will prove the equivalence of the equations~\eqref{eq:gC} and \eqref{eq:Cs_modified} for any 
  $\ell =0,1, \ldots, \sm$, by means of induction on~$\ell$. The equivalence of both equations trivially holds for $\ell=0$ and can be directly 
  obtained for $\ell =1$ by applying \eqref{eq:a10b10} in equation~\eqref{eq:Cs_modified}.
  We will assume now that the equivalence of the two equations~\eqref{eq:gC} and \eqref{eq:Cs_modified} holds for all $\ell \leq \sm-1$, 
  and we will show it for $\ell=\sm$. 
 Using the induction assumption, equation \eqref{eq:gC2} implies that
 \begin{equation}   \label{eq:partial1gL_IA}
  \partial_1^\ell \g^{(\LL)}(0,\xi) =  \sum_{i=0}^{\ell} {\ell \choose i} 
  \left(\beta^{(\LL)}(\xi)\right)^{\ell-i} \left(\alpha^{(\LL)}(\xi)\right)^{i}  \dd^{\ell-i} \gC_i^{}(\xi),
  \quad 1\leq \ell \leq \sm-1.
 \end{equation}
Further,
differentiating \eqref{eq:partial1gL_IA} with respect to the second 
argument yields
 \begin{equation}  \label{eq:partial1partial2gL_IA}
 \partial_2 ^j  \, \partial_1^\ell \g^{(\LL)}(0,\xi) =  \sum_{i=0}^{\ell}  \sum_{\rho=0}^j  {\ell \choose i} {j \choose \rho} 
 \, \dd^{j-\rho} \left( \left(\beta^{(\LL)}(\xi)\right)^{\ell-i} \left(\alpha^{(\LL)}(\xi)\right)^{i}\right)^{}  \dd^{\ell-i+\rho}\gC_i^{}(\xi),
 \end{equation}
 for $1\leq \ell \leq \sm-1$ and $j\geq 0$. 
 In the following, we will skip the arguments in order to simplify the expressions. Using \eqref{eq:FaaBruno} and \eqref{eq:partial1partial2gL_IA}, 
 equation~\eqref{eq:Cs_modified} is equivalent to
 \begin{align}
  &  0= \frac{\partial_1^\sm \g^{(\RR)}}{ (\alpha^{(\RR)})^{\sm}} - \frac{\partial_1^\sm \g^{(\LL)}}{ (\alpha^{(\LL)})^{\sm}} - 
  \sum_{|\bfm{\sigma}|  
  = 1 \atop \sigma_1 < \sm }^{\sm} \frac{\AS_{\bfm{\sigma};\sm}}{ (\alpha^{(\RR)})^{\sm}} \, 
   \sum_{i=0}^{\sigma_1}  \sum_{\rho=0 
   }^{\sigma_2}  {\sigma_1 \choose i} {\sigma_2 \choose \rho} 
 \,\dd^{\sigma_2-\rho} \left( \left(\beta^{(\LL)}\right)^{\sigma_1-i} \left(\alpha^{(\LL)}\right)^{i}\right)^{}  \dd^{\sigma_1-i+\rho} \gC_i^{} 
 \nonumber \\
  & = \frac{\partial_1^\sm \g^{(\RR)}}{ (\alpha^{(\RR)})^{\sm}} - \frac{\partial_1^\sm \g^{(\LL)}}{ (\alpha^{(\LL)})^{\sm}} - 
  \sum_{\sigma_1  = 0 }^{\sm-1} \sum_{\sigma_2=0}^{\sm-\sigma_1} \frac{\AS_{\bfm{\sigma};\sm}}{ (\alpha^{(\RR)})^{\sm}}  \, 
   \sum_{i=0}^{\sigma_1}  \sum_{\rho=0 
   }^{\sigma_2}  {\sigma_1 \choose i} {\sigma_2 \choose \rho} 
 \,\dd^{\sigma_2-\rho} \left( \left(\beta^{(\LL)}\right)^{\sigma_1-i} \left(\alpha^{(\LL)}\right)^{i}\right)^{}  \dd^{\sigma_1-i+\rho}\gC_i^{} 
 \nonumber \\
  & = \frac{\partial_1^\sm \g^{(\RR)}}{ (\alpha^{(\RR)})^{\sm}} - \frac{\partial_1^\sm \g^{(\LL)}}{ (\alpha^{(\LL)})^{\sm}} - 
  \sum_{i=0}^{\sm - 1}  \sum_{\sigma_1  = i }^{\sm-1} \sum_{\sigma_2=0}^{\sm-\sigma_1}   \sum_{\rho=0 
   }^{\sigma_2} \frac{\AS_{\bfm{\sigma};\sm}}{ (\alpha^{(\RR)})^{\sm}}  \,  {\sigma_1 \choose i} {\sigma_2 \choose \rho} 
 \,\dd^{\sigma_2-\rho} \left( \left(\beta^{(\LL)}\right)^{\sigma_1-i} \left(\alpha^{(\LL)}\right)^{i}\right)^{} \dd^{\sigma_1-i+\rho} \gC_i^{} .
 \label{eq:Cs_modified2} 
 \end{align}
 It is straightforward to see that $\sigma_1-i+\rho \leq s-i$. Therefore, we can write equation~\eqref{eq:Cs_modified2} also as
\begin{equation}  \label{eq:esCinu}
 0 =  \frac{\partial_1^\sm \g^{(\RR)}}{ (\alpha^{(\RR)})^{\sm}} - \frac{\partial_1^\sm \g^{(\LL)}}{ (\alpha^{(\LL)})^{\sm}} - 
  \sum_{i=0}^{\sm - 1}  \sum_{j  = 0 }^{\sm-i}   \Coeff_{i,j}^\sm  \, \dd^{j} \gC_i^{},
\end{equation}
where $\Coeff_{i,j}^\sm$ are univariate functions.
  In order to prove the theorem,  which means now to demonstrate the equivalence of the equations~\eqref{eq:gC} and \eqref{eq:esCinu}, it 
  remains to show that these functions are given as
  \begin{equation}
     \Coeff_{i,\sm-i}^\sm  = {\sm \choose i}  \left(\left(\frac{\beta^{(\RR)}}{\alpha^{(\RR)}}\right)^{\sm-i} - 
     \left(\frac{\beta^{(\LL)}}{\alpha^{(\LL)}}\right)^{\sm-i}\right),  \quad i=0,1,\ldots,\sm-1, \label{eq:Cf1}
  \end{equation}
  and
  \begin{equation}
     \Coeff_{i,j}^\sm  = 0,\quad j = 0,1,\ldots,\sm-i-1,  \quad i=0,1,\ldots,\sm-1. \label{eq:Cf2}
 \end{equation}
 
We will first consider the case $i+j = \sm$, and will hence prove formula \eqref{eq:Cf1}. Since $\rho = \sm - \sigma_1$, 
$\rho \leq \sigma_2$ and $| \bfm{\sigma}|\leq \sm$, it follows that $\rho=\sigma_2=s-\sigma_1$. Then, we have
\begin{align*}
 \Coeff_{i,\sm-i}^\sm & = \sum_{\sigma_1 = i}^{\sm-1}
 \frac{\AS_{(\sigma_1,\sm-\sigma_1);\sm}}{ (\alpha^{(\RR)})^{\sm}} 
     {\sigma_1 \choose i} \left(\beta^{(\LL)}\right)^{\sigma_1-i} \left(\alpha^{(\LL)}\right)^{i}  \\
     & = \frac{1}{ (\alpha^{(\RR)})^{\sm}}  \sum_{\sigma_1 = i}^{\sm-1}
 {\sm \choose \sigma_1} {\sigma_1 \choose i} \A_{1,0}^{\sigma_1} \, \B_{1,0}^{\sm-\sigma_1} \, \left(\beta^{(\LL)}\right)^{\sigma_1-i} 
 \left(\alpha^{(\LL)}\right)^{i} \\
 & = \frac{1}{ (\alpha^{(\RR)})^{\sm}}  \left(\frac{\alpha^{(\LL)}}{\beta^{(\LL)}}\right)^{i} \; \sum_{j=0}^{\sm-i-1} {\sm \choose j+i} 
 {j+i \choose i} \left(\A_{1,0} \beta^{(\LL)} \right)^{j+i} \, \B_{1,0}^{\sm-(i+j)} \\
 & = \frac{1}{ (\alpha^{(\RR)})^{\sm}}  \left(\frac{\alpha^{(\LL)}}{\beta^{(\LL)}}\right)^{i} \; \sum_{j=0}^{\sm-i-1} {\sm-i \choose j} 
 {\sm \choose i} \left(\A_{1,0} \beta^{(\LL)} \right)^{j+i} \, \B_{1,0}^{\sm-(i+j)} \\
  & = \frac{1}{ (\alpha^{(\RR)})^{\sm}}  \left(\frac{\alpha^{(\LL)}}{\beta^{(\LL)}}\right)^{i} \left(\A_{1,0} \beta^{(\LL)} \right)^{i}  {\sm \choose i} 
 \left(  \sum_{j=0}^{\sm-i} {\sm-i \choose j} \left(\A_{1,0} \beta^{(\LL)} \right)^{j} \, \B_{1,0}^{(\sm-i)-j} -  \left(\A_{1,0} \beta^{(\LL)} \right)^{\sm-i}\right) \\
 & = \frac{1}{ (\alpha^{(\RR)})^{\sm}}  \left(\frac{\alpha^{(\LL)}}{\beta^{(\LL)}}\right)^{i} \left(\A_{1,0} \beta^{(\LL)} \right)^{i}  {\sm \choose i} 
 \left(  \left(\A_{1,0} \beta^{(\LL)} +\B_{1,0}\right)^{\sm-i}  -  \left(\A_{1,0} \beta^{(\LL)} \right)^{\sm-i}\right).
\end{align*}
Due to $\A_{1,0} \beta^{(\LL)}  = \alpha^{(\RR)} \beta^{(\LL)} \left(\alpha^{(\LL)}\right)^{-1} $ and 
$\A_{1,0} \beta^{(\LL)} +\B_{1,0} = \beta^{(\RR)} $, we further obtain
\begin{equation*}
 \Coeff_{i,\sm-i}^\sm  =  {\sm \choose i}  \frac{1}{ (\alpha^{(\RR)})^{\sm-i}}  \left( 
  \left( \beta^{(\RR)}\right)^{\sm-i} - \left(  \frac{\alpha^{(\RR)} \beta^{(\LL)}}{\alpha^{(\LL)}}\right)^{\sm-i}  \right) = {\sm \choose i}  \left( \left(  \frac{\beta^{(\RR)}}{ \alpha^{(\RR)}}\right)^{\sm-i} -  \left(\frac{\beta^{(\LL)}}{ 
  \alpha^{(\LL)}}\right)^{\sm-i} \right) .
 \end{equation*}
 
Let us now consider \eqref{eq:Cf2}. By fixing $i$ and $j \in \{0,1,\ldots,\sm-i-1 \}$, we obtain $\rho = i+j -\sigma_1$. 
Since $\rho \geq 0$, we further get $\sigma_1 \leq i+j \leq s-1$, and we have
\begin{align}  \label{eq:C_fi_general}
 \Coeff_{i,j}^\sm & =  \frac{1}{ (\alpha^{(\RR)})^{\sm}} \sum_{\sigma_1=i}^{i+j} \; \sum_{\sigma_2 = i+j-\sigma_1}^{s-\sigma_1}
\AS_{\bfm{\sigma};\sm}  \,  {\sigma_1 \choose i} {\sigma_2 \choose i+j-\sigma_1} 
  \dd^{\sigma_2 - ( i+j - \sigma_1)} \left( \left(\beta^{(\LL)}\right)^{\sigma_1-i} \left(\alpha^{(\LL)}\right)^{i}\right)^{} \nonumber \\
  = & \frac{1}{ (\alpha^{(\RR)})^{\sm}} \sum_{\mu_1=0}^{j} \; \sum_{\mu_2 = 0}^{s - (i+j)}
\AS_{(\mu_1+i, \mu_2+(j-\mu_1));\sm}  \,  {\mu_1 +i \choose i} {\mu_2 + (j-\mu_1)\choose j-\mu_1} 
  \dd^{\,\mu_2}\left( \left(\beta^{(\LL)}\right)^{\mu_1} \left(\alpha^{(\LL)}\right)^{i}\right)^{} . 
 \end{align}
 By using the recursion in Theorem~\ref{thm1} step by step $\sm-(i+j)$ times, expression \eqref{eq:C_fi_general} simplifies to
\begin{align}
    \Coeff_{i,j}^\sm   = & \, 
    \frac{  (-1)^{\sm - (i + j)} \sm (\sm -1-i)!  {\sm-1 \choose i}  \left(\beta^{(\RR)}\right)^{j-1} \left(\alpha^{(\RR)}\right)^{i-1} 
    }{ j! \left(\alpha^{(\RR)}\right)^{\sm}\left(\alpha^{(\LL)}\right)^{2(\sm-(i+j))}} \left( i \,\beta^{(\RR)} \left(\alpha^{(\LL)}\right)'+j\, 
    \alpha^{(\RR)} \left(\beta^{(\LL)}\right)' \right) \nonumber \\ 
    & \cdot \left( \alpha^{(\LL)} \beta^{(\RR)} \left(\alpha^{(\LL)}\right)' + \alpha^{(\RR)} \left(\alpha^{(\LL)} \left(\beta^{(\LL)}\right)' - 
    2 \beta^{(\LL)} \left(\alpha^{(\LL)}\right)' \right) \right)^{\sm-1-(i+j)} \label{eq:factor} \\ 
    & \cdot 
    \left( \left(\alpha^{(\RR)} \right)^2 \beta^{(\LL)} - \alpha^{(\LL)} \alpha^{(\RR)} \beta^{(\RR)} +
    \left(\alpha^{(\LL)} \right)^2 \A_{1,0} \,  \B_{1,0} \right). \nonumber
\end{align}
Since the last factor of \eqref{eq:factor} is equal to zero, i.e.
$$
 \left(\alpha^{(\RR)} \right)^2 \beta^{(\LL)} - \alpha^{(\LL)} \alpha^{(\RR)} \beta^{(\RR)} +
    \left(\alpha^{(\LL)} \right)^2 \A_{1,0} \,  \B_{1,0} = 
    - \alpha^{(\RR)}  \left( \alpha^{(\LL)} \beta^{(\RR)} - \alpha^{(\RR)}  \beta^{(\LL)}\right) +
    \alpha^{(\RR)} \beta = 0,
$$
relation \eqref{eq:Cf2} holds. Employing \eqref{eq:Cf1} and \eqref{eq:Cf2}, we can now simplify equation \eqref{eq:esCinu} to
 $$
 \frac{\partial_1^\sm \g^{(\RR)}}{ (\alpha^{(\RR)})^{\sm}} - \frac{\partial_1^\sm \g^{(\LL)}}{ (\alpha^{(\LL)})^{\sm}} - 
  \sum_{i=0}^{\sm - 1}  {\sm \choose i} \left(\frac{\beta^{(\RR)}}{\alpha^{(\RR)}}\right)^{\sm-i} 
  \hskip-0.6em 
  \dd^{\sm-i} \gC_i^{} +  \sum_{i=0}^{\sm - 1} 
      {\sm \choose i} \left(\frac{\beta^{(\LL)}}{\alpha^{(\LL)}}\right)^{\sm-i}  \hskip-0.6em \dd^{\sm-i} \gC_i^{}  =0,
 $$
 which is equivalent to equation~\eqref{eq:gC}, and which finally concludes the proof.
 \end{proof}

The $C^s$-smooth isogeometric spline space~$\mathcal{V}^{\sm}$ over a bilinearly parameterized two-patch domain $\overline{\Omega} = 
\overline{\Omega^{(\LL)}} \cup \overline{\Omega^{(\RR)}}$ can be decomposed into the direct sum of three subspaces, namely
$$
\mathcal{V}^{\sm} = \mathcal{V}_{\Omega^{(\LL)}}^{\sm} \oplus \mathcal{V}_{\Omega^{(\RR)}}^{\sm}  \oplus \mathcal{V}_{\Gamma}^{\sm},
$$
where the subspaces~$\mathcal{V}_{\Omega^{(\Side)}}^{\sm}$, $\Side \in \{\LL,\RR\}$, and $\mathcal{V}_{\Gamma}^{\sm}$ are given by
\[
 \mathcal{V}_{\Omega^{(\Side)}}^{\sm} =  \left\{ \phi \in \mathcal{V}^{\sm} \; | \; \g^{(\Side)}(\xi_1,\xi_2)= \sum_{j_1=\sm+1}^{n-1}\sum_{j_2=0}^{n-1} 
 \coe^{(\Side)}_{j_1,j_2} N_{j_1,j_2}^{\ab{p},\ab{r}}(\xi_1,\xi_2), \mbox{ } \g^{(\widetilde{\Side})}(\xi_1,\xi_2)=0, 
 \widetilde{\Side} \neq \Side \right\}
\]
and
\[
 \mathcal{V}_{\Gamma}^{\sm} = \left\{\phi \in \mathcal{V}^{\sm} \; | \; \g^{(\Side)}(\xi_1,\xi_2)= \sum_{j_1=0}^{\sm}\sum_{j_2=0}^{n-1} 
 \coe^{(\Side)}_{j_1,j_2} N_{j_1,j_2}^{\ab{p},\ab{r}}(\xi_1,\xi_2), \mbox{ } \Side \in \{\LL,\RR \} 
\right\},
\]
respectively. The subspaces~$\mathcal{V}_{\Omega^{(\Side)}}^{\sm}$, $\Side \in \{\LL,\RR\}$, can be simply described as
\begin{align*}
\mathcal{V}^{\sm}_{\Omega^{(\Side)}} &= \Span \left\{ \phi_{\Omega^{(\Side)}; j_1,j_2} |\; j_1=\sm+1,\ldots,n-1,\; j_2=0,1,\ldots, n-1 \right\},\quad \Side 
\in \{\LL,\RR\},
\end{align*}
 with the functions
\begin{equation}  \label{eq:PhiOmega}
{\phi}_{\Omega^{(\Side)};  j_1,j_2}(\bfm{x})  = 
\begin{cases}
   (N_{  j_1,j_2}^{\bfm{p},\bfm{r}}\circ (\ab{F}^{(\Side)})^{-1})(\bfm{x}) \;
\mbox{ if }\f \, \bfm{x} \in \overline{\Omega^{(\Side)}},
\\ 0 \quad \mbox{ if }\f \, \bfm{x} \in \overline{\Omega} \backslash \overline{\Omega^{(\Side)}}.
\end{cases} 
\end{equation}
Since the functions $\phi_{\Omega^{(\Side)}; j_1,j_2}$, $j_1=\sm+1,\ldots,n-1$, $j_2=0,1,\ldots, n-1$, are just ``standard'' isogeometric spline 
functions of at least $C^s$-continuity,
they are linearly independent, and therefore form a basis of the space~$\mathcal{V}^{\sm}_{\Omega^{(\Side)}}$.
The following theorem specifies now an explicit representation of an isogeometric function~$\phi \in \mathcal{V}^{\sm}_{\Gamma}$.
\begin{theorem} \label{lem:explicit_second}
 Let $\phi \in \mathcal{V}^{s}_{\Gamma}$, then the two associated spline functions $\g^{(\Side)} = \phi \circ \ab{F}^{(\Side)}$, 
 $\Side \in \{\LL,\RR \}$, can be represented as
 \begin{equation}  \label{eq:presentation}
 \g^{(\Side)}(\xi_1,\xi_2) = \sum_{i=0}^{\sm} \left( \sum_{j=0}^i {i \choose j} \left( \beta^{(\Side)}(\xi_2)\right)^{i-j} 
 \left( \alpha^{(\Side)}(\xi_2)\right)^{j} \dd^{i-j}\gC_j^{}(\xi_2)  \right) M_i^{p,r}(\xi_1)
\end{equation}
where the functions $\gC_{j}$ are defined in \eqref{eq:gC} and \eqref{eq:gC2}, and the functions $M_{i}^{p,r}$ are given as
\begin{equation*}  \label{eq:M012}
  M_i^{p,r}(\xi) = \sum_{j=i}^\sm \frac{ {j \choose i} h^i}{\prod_{\ell=0}^{i-1} (p-\ell)} N_j^{p,r}(\xi) , \quad i = 0,1,\ldots,\sm.
\end{equation*}
\end{theorem}
\begin{proof}
By means of the Taylor expansion of $\g^{(\Side)}(\xi_1,\xi_2)$ at $(\xi_1,\xi_2)=(0,\xi_2)$, and due to Theorem~\ref{thm:g_ell},
we obtain that
\begin{align*}  
   \g^{(\Side)}(\xi_1,\xi_2) = & \; \g^{(\Side)}(0,\xi_2) + \partial_1 \g^{(\Side)} (0,\xi_2) \xi_1  + \partial_1^2 \g^{(\Side)} (0,\xi_2) 
   \frac{\xi_1^2}{2} + \ldots + \partial_1^\sm \g^{(\Side)} (0,\xi_2) \frac{\xi_1^\sm}{\sm !} + \mathcal{O}(\xi_1^{\sm+1})  \nonumber \\
  = &\; \gC_0(\xi_2)  + \left( \alpha^{(\Side)}(\xi_2) \gC_1(\xi_2) + \beta^{(\Side)}(\xi_2) \gC_0'(\xi_2)\right) \xi_1 + \nonumber \\
  &\left( (\alpha^{(\Side)})^2(\xi_2)\gC_2(\xi_2) +  2\beta^{(\Side)} \alpha^{(\Side)} (\xi_2) \gC_1'(\xi_2)  + (\beta^{(\Side)})^2(\xi_2) 
  \gC_0''(\xi_2)\right) \frac{\xi_1^2}{2}+ \ldots + \nonumber \\ 
  &\left( \sum_{j=0}^\sm {\sm \choose j} \left( \beta^{(\Side)}(\xi_2)\right)^{\sm-j} \left( \alpha^{(\Side)}(\xi_2)\right)^{j} 
  \dd^{\sm-j}\gC_j^{}(\xi_2) \right) \frac{\xi_1^\sm}{\sm!} + \mathcal{O}(\xi_1^{\sm+1}). \nonumber
\end{align*}
Using the fact that the functions $\g^{(\Side)}(\xi_1,\xi_2)$, $\Side \in \{\LL,\RR\}$, possess just a spline representation of the form
\[
 \g^{(\Side)}(\xi_1,\xi_2)= \sum_{j_1=0}^{\sm}\sum_{j_2=0}^{n-1} \coe^{(\Side)}_{j_1,j_2} N_{j_1,j_2}^{\ab{p},\ab{r}}(\xi_1,\xi_2),
\]
we further get
  \begin{align*}
 \g^{(\Side)}(\xi_1,\xi_2)  = & \; \gC_0(\xi_2) M_0^{p,r}(\xi_1) + \left( \alpha^{(\Side)}(\xi_2) \gC_1(\xi_2) + \beta^{(\Side)}(\xi_2) 
 \gC_0'(\xi_2)\right) M_1^{p,r}(\xi_1) + 
  \nonumber \\
  &\left( (\alpha^{(\Side)})^2(\xi_2)\gC_2(\xi_2) +  2\beta^{(\Side)} \alpha^{(\Side)} (\xi_2) \gC_1'(\xi_2)  + (\beta^{(\Side)})^2(\xi_2) 
  \gC_0''(\xi_2)\right) M_2^{p,r}(\xi_1)+ \ldots + \nonumber \\ 
  &\left( \sum_{j=0}^\sm {\sm \choose j} \left( \beta^{(\Side)}(\xi_2)\right)^{\sm-j} \left( \alpha^{(\Side)}(\xi_2)\right)^{j} \dd^{\sm-j}
  \gC_j^{}(\xi_2) \right) M_{\sm}^{p,r}(\xi_1),
\end{align*}
with
$$
  M_i^{p,r}(\xi_1) = \sum_{j=0}^\sm \lambda_{i,j} N_j^{p,r}(\xi_1), \quad i=0,1,\ldots,\sm.
$$
The unknown parameters $\lambda_{i,j}$, $i,j=0,1,\ldots,\sm$, are then determined by the conditions 
$
\dd^{\ell} M_i^{p,r}(0)=\delta_{i,\ell},
$
where $\delta_{i,\ell}$ is the Kronecker delta. By the properties of the B-splines $N_j^{p,r}$, we obtain for the unknowns $\lambda_{i,j}$ 
the following system of $2(\sm+1)$ 
equations
$$
 \dd^{\ell} M_i^{p,r}(0) = h^{-\ell} \left( \prod_{\rho=0}^{\ell-1} (p-\rho)\right) \sum_{\rho=0}^{\ell} (-1)^{\ell-\rho} {\ell \choose \rho} 
 \lambda_{i,\rho}, \quad i,j=0,1,\ldots,\sm,
$$ 
which possesses the solution
$$
\lambda_{i,0} = \lambda_{i,1} = \cdots = \lambda_{i,i-1} =0, \quad
\lambda_{i,j} = \frac{ {j \choose i} h^i }{\prod_{\rho = 0}^{i-1} (p-\rho)}, \; j \geq i. 
$$
\end{proof}
 
Any $(\sm+1)$-tuple of functions $(\gC_0, \gC_1,\ldots,\gC_\sm)$, which assures after inserting in \eqref{eq:presentation}, 
that the two functions $\g^{(\Side)} = \phi \circ \ab{F}^{(\Side)}$, $\Side \in \{\LL,\RR \}$, belong to the spline 
space~$\mathcal{S}_{h}^{\ab{p},\ab{r}}([0,1]^{2})$, defines an isogeometric function $\phi \in \mathcal{V}_\Gamma^{\sm}$. 
One can easily verify by means of representation~\eqref{eq:presentation}, that the $(\sm+1)$-tuples formed by
\begin{equation}  \label{eq:basisG}
(\gC_0,\gC_1,\ldots,\gC_\sm) = \left(\underbrace{0,0,\ldots,0}_{j_1}, N_{j_2}^{p-j_1,r+\sm-j_1},\underbrace{0,\ldots,0}_{\sm-j_1}\right) ,\; 
j_2=0,1,\ldots,n_{j_1}-1, \; j_1=0,1,\ldots,\sm,
\end{equation}
with $n_{j_1} = \dim (\mathcal{S}_{h}^{{p-j_1},{r+\sm-j_1}}([0,1]^{2})) = p+1-j_1+k(p-r-\sm+j_1)$,
yield isogeometric functions $\phi \in \mathcal{V}_\Gamma^{\sm}$. Let us denote these functions by $\phi_{\Gamma;j_1,j_2}$. However, 
linear combinations of the functions $\phi_{\Gamma;j_1,j_2}$ are not necessarily the only functions in $\mathcal{V}_\Gamma^{\sm}$,  and 
the construction of a basis for the space~$\mathcal{V}^{\sm}_{\Gamma}$ is now a very challenging task, which requires the study of a lot of 
different possible cases, cf.~\cite{KaSaTa17a} and \cite{KaVi17c} for $\sm=1$ and $\sm=2$, respectively. This is a direct consequence of the 
fact that the dimension of the space~$\mathcal{V}^{\sm}_{\Gamma}$ heavily depends on the parameterization of the initial geometry. Therefore, we 
consider instead of the entire space~$\mathcal{V}^{\sm}_{\Gamma}$ the subspace
$$
 \widetilde{\mathcal{W}}_\Gamma^{\sm} = \Span \left\{ \phi_{\Gamma; j_1,j_2} |\; j_1=0,1,\ldots,\sm, \; j_2=0,1,\ldots,n_{j_1} -1 \right\} 
 \subseteq \mathcal{V}_\Gamma^{\sm},
$$
which further leads to the subspace~$\mathcal{W}^{\sm} \subseteq \mathcal{V}^{\sm}$ given as
\[
\mathcal{W}^{\sm} = \mathcal{V}_{\Omega^{(\LL)}}^{\sm} \oplus \mathcal{V}_{\Omega^{(\RR)}}^{\sm}  
\oplus \widetilde{\mathcal{W}}_{\Gamma}^{\sm}.
\]
The selection of this subspace is motivated by the numerical results in \cite{KaVi17c} for $\sm=2$, and by our numerical experiments in 
Section~\ref{sec:examples} for $s=1,\ldots,4$, 
which indicate that the subspace~$\mathcal{W}^{\sm}$
possesses as the entire space~$\mathcal{V}^{\sm}$ optimal approximation properties.
The functions $\phi_{\Gamma; j_1,j_2}$ can be expressed by inserting 
\eqref{eq:basisG} into \eqref{eq:presentation} as
\begin{equation}
 \phi_{\Gamma; j_1,j_2}(\bfm{x}) = 
 \begin{cases}
  (\g_{\Gamma; j_1,j_2}^{(\LL)} \circ (\ab{F}^{(\LL)})^{-1})(\bfm{x}) \;
\mbox{ if }\f \, \bfm{x} \in \overline{\Omega^{(\LL)}},
\\[0.15cm] 
  (\g_{\Gamma; j_1,j_2}^{(\RR)} \circ (\ab{F}^{(\RR)})^{-1} )(\bfm{x}) \;
\mbox{ if }\f \, \bfm{x} \in \overline{\Omega^{(\RR)}}, 
\end{cases}
\, j_1=0,1,\ldots,\sm, \; j_2=0,1,\ldots,n_{j_1}-1, 
 \label{eq:basisFunctionsGenericEdge}
\end{equation}
where the functions $\g_{\Gamma; j_1,j_2}^{(\Side)}$
$ \Side \in \{\LL,\RR \}$, are given by
\begin{equation}   \label{eq:basisFunctionsGenericG}
 \g_{\Gamma; j_1,j_2}^{(\Side)}(\xi_1,\xi_2) =  \sum_{i=j_1}^{\sm}  {i \choose j_1} \left( \beta^{(\Side)}(\xi_2)\right)^{i-j_1} 
 \left( \alpha^{(\Side)}(\xi_2)\right)^{j_1} \dd^{i-j_1}\left( N_{j_2}^{p-j_1,r+\sm-j_1}(\xi_2)\right) \,  M_i^{p,r} (\xi_1).
\end{equation}
In addition, the functions~$\phi_{\Gamma; j_1,j_2}$, $j_1=0,1,\ldots,\sm$, $j_2=0,1,\ldots,n_{j_1} -1$, are linearly independent by construction, 
and hence form a basis of the space~$\widetilde{\mathcal{W}}_\Gamma^{\sm}$ and together with the functions $\phi_{\Omega^{(\Side)};j_1,j_2}$, 
$j_1=\sm+1,\ldots,n-1$, $j_2=0,1,\ldots, n-1$, $\Side \in \{\LL,\RR \}$, a basis of the space~$\mathcal{W}^{\sm}$. This further implies that the dimension of 
the subspace~$\mathcal{W}^{\sm}$ is independent of the parameterization of the initial geometry, and is equal to
\[
 \dim \mathcal{W}^{\sm} = \dim \mathcal{V}_{\Omega^{(\LL)}}^{\sm} + \dim \mathcal{V}_{\Omega^{(\RR)}}^{\sm} +  
 \dim \widetilde{\mathcal{W}}_\Gamma^{\sm},
\]
with
\[
 \dim \mathcal{V}_{\Omega^{(\LL)}}^{\sm}  = \dim \mathcal{V}_{\Omega^{(\RR)}}^{\sm} = n(n-(\sm+1))
\quad 
{\rm and} \quad
  \dim \widetilde{\mathcal{W}}_\Gamma^{\sm}  = \sum_{j_1=0}^\sm (n_{j_1}+1).
\]

\section{$C^{\sm}$-smooth isogeometric spaces over multi-patch domains} \label{sec:design_multipatches}

In this section, we will extend the construction of the $C^s$-smooth isogeometric subspace~$\mathcal{W}^s \subseteq \mathcal{V}^s$ 
for bilinearly parameterized two-patch domains, which has been described in the previous section, to the case of bilinear multi-patch 
domains~$\overline{\Omega}$ with more than two patches and with possibly extraordinary vertices.
The proposed construction will work uniformly for all possible multi-patch configurations and is much simpler as for 
the entire $C^{\sm}$-smooth space~$\mathcal{V}^{\sm}$. Thereby, the design of the subspace~$\mathcal{W}^{\sm}$ will be based on the results 
of the two-patch case, and is motivated by the methods~\cite{KaSaTa19b,KaSaTa19a} and \cite{KaVi19a,KaVi20}, where similar 
subspaces have been generated for a global smoothness of~$\sm=1$ and $\sm=2$, respectively. There, it has been numerically shown that the 
corresponding subspaces possess as the entire $C^{\sm}$-smooth space~$\mathcal{V}^{\sm}$ optimal approximation properties. This will be also 
numerically verified in Section~\ref{sec:examples} on the basis of an example for the subspace~$\mathcal{W}^{\sm}$ for $\sm=1,\ldots,4$.

The subspace~$\mathcal{W}^{\sm}$ will be generated as the
direct sum of smaller subspaces corresponding to the single patches~$\Omega^{(i)}$, $i \in \mathcal{I}_\Omega$, edges~$\Gamma^{(i)}$, 
$i \in \mathcal{I}_\Gamma$, and vertices $\Xi^{(i)}$, $i \in \mathcal{I}_\Xi$, i.e., 
\begin{equation} \label{eq:W}
    \mathcal{W}^\sm =  \left(\bigoplus_{i \in \mathcal{I}_\Omega} \mathcal{W}^\sm_{\Omega^{(i)}}\right) \oplus 
    \left(\bigoplus_{i \in \mathcal{I}_\Gamma} \mathcal{W}^\sm_{\Gamma^{(i)}}\right) 
    \oplus  
    \left(\bigoplus_{i \in \mathcal{I}_\Xi} \mathcal{W}^\sm_{\Xi^{(i)}}\right).
\end{equation}
In order to ensure $h$-refineable and well-defined subspaces, we 
have to assume additionally that the number of inner knots satisfies 
$k \geq \frac{4\sm+1-p}{p-r-\sm}$, which implies $h \leq \frac{p-r-s}{3s-r+1}$. The construction of the particular subspaces 
in \eqref{eq:W} will be based on functions from the subspaces 
$\mathcal{V}^\sm_{\Omega^{(\Side)}}$, $\Side \in \{\LL,\RR \}$, and ${\mathcal{W}^\sm_{\Gamma}}$, 
which have been both generated in Section~\ref{sec:two-patch-case} for the case of a bilinearly parameterized two-patch domain
$\overline{\Omega} = \overline{\Omega^{(\LL)}} \cup \overline{\Omega^{(\RR)}}$ with the common edge $\overline{\Gamma} = 
\overline{\Omega^{(\LL)}} \cap 
\overline{\Omega^{(\RR)}}$, and will be described in detail below. 

\subsection{The patch and edge subspaces}

We will first describe the construction of the subspaces $\mathcal{W}^\sm_{\Omega^{(i)}}$, $i \in \mathcal{I}_\Omega$, and 
$\mathcal{W}^\sm_{\Gamma^{(i)}}$, $i \in \mathcal{I}_\Gamma$.
Analogous to \eqref{eq:PhiOmega} in case of a two-patch domain, we define the functions ${\phi}_{\Omega^{(i)};j_1,j_2}: \overline{\Omega} 
\to \R$ in case of a multi-patch domain $\overline{\Omega} = \cup_{i \in \mathcal{I}_{\Omega}} \overline{\Omega^{(i)}}$ as
\begin{equation}  \label{eq:PhiOmega2}
{\phi}_{\Omega^{(i)};j_1,j_2}(\bfm{x})  = 
\begin{cases}
   (N_{j_1,j_2}^{\bfm{p},\bfm{r}}\circ (\ab{F}^{(i)})^{-1})(\bfm{x}) \;
\mbox{ if }\f \, \bfm{x} \in \overline{\Omega^{(i)}},
\\ 0 \quad \mbox{ if }\f \, \bfm{x} \in \overline{\Omega} \backslash \overline{\Omega^{(i)}},
\end{cases} 
\end{equation}
and then define the patch subspace $\mathcal{W}^\sm_{\Omega^{(i)}}$ as
\begin{equation*} \label{eq:spaceW0hOmega}
\mathcal{W}^\sm_{\Omega^{(i)}} =  \Span \left\{ \phi_{\Omega^{(i)};j_1,j_2} |\;  j_1,j_2=\sm+1,\sm+2,\ldots,n-1-(\sm+1)  \right\}.
\end{equation*}
We clearly have
$
\mathcal{W}^\sm_{\Omega^{(i)}}  \subseteq \mathcal{V}^\sm,
$
since the functions~$\phi_{\Omega^{(i)};j_1,j_2}$, $j_1,j_2=\sm+1,\ldots,n-1-(\sm+1)$, have a support entirely inside 
$ \Omega^{(i)}$, are clearly $C^\sm$-smooth on $\Omega^{(i)}$, and have vanishing values and derivatives up to order $\sm$ on $\partial 
\Omega^{(i)}$.

Let us construct now the edge subspaces $\mathcal{W}^\sm_{\Gamma^{(i)}}$, $i \in \mathcal{I}_\Gamma$, where we have to distinguish 
between boundary and inner edges~$\Gamma^{(i)}$. In case of a boundary edge $\Gamma^{(i)} \subseteq \overline{\Omega^{(i_0)}}$, 
$i_0 \in \mathcal{I}_{\Omega}$, we can assume without loss of generality that the boundary edge $\Gamma^{(i)}$ is given by
$\ab{F}^{(i_0)}(\{0 \} \times (0,1)) $.
Then, we generate the edge subspace $\mathcal{W}^\sm_{\Gamma^{(i)}}$ as
\begin{align*} 
\mathcal{W}^\sm_{\Gamma^{(i)}} =&  \Span \left\{ \phi_{\Omega^{(i_0)}; j_1,j_2} |\; j_2=2\sm +1-j_1,\ldots,n+j_1-(2\sm +2), \; j_1=0,1,\ldots,\sm \right\} ,
\end{align*}
where the functions~$\phi_{\Omega^{(i_0)}; j_1,j_2}$ are defined as in~\eqref{eq:PhiOmega2}. Similar to the
patch subspace~$\mathcal{W}^\sm_{\Omega^{(i)}}$, the functions 
$\phi_{\Omega^{(i_0)}; j_1,j_2}$, $j_2=2\sm +1-j_1,\ldots,n+j_1-(2\sm +2)$, $j_1=0,1,\ldots,\sm$, are trivially $C^{\sm}$-smooth on 
$\overline{\Omega}$, which implies that $\mathcal{W}^\sm_{\Gamma^{(i)}}  \subseteq \mathcal{V}^\sm $.

Let us consider now the case of an inner edge $\Gamma^{(i)} \subseteq  \overline{\Omega^{(i_0)}} \cap \overline{\Omega^{(i_1)}}$, 
$i_0,i_1 \in \mathcal{I}_{\Omega}$. Without loss of generality, we can assume that the two associated geometry 
mappings~$\ab{F}^{(i_0)}$ and $\ab{F}^{(i_1)}$ are parameterized as shown in Fig.~\ref{fig:twopatchCase}.
The edge subspace $\mathcal{W}^\sm_{\Gamma^{(i)}}$ is now defined as
\begin{equation*} \label{eq:spaceW0hGamma}
\mathcal{W}^\sm_{\Gamma^{(i)}} =  \Span \left\{\phi_{\Gamma^{(i)};j_1,j_2}| \;\; j_2=2\sm+1 -j_1,
\ldots,n_{j_1}+j_1-(2\sm+2),\; j_1=0,1,\ldots,\sm \right\},
\end{equation*} 
where the functions $\phi_{\Gamma^{(i)};j_1,j_2}$ possess similar to the two-patch case~\eqref{eq:basisFunctionsGenericEdge} the form
\begin{equation}  \label{eq:defphiGamma}
    \phi_{\Gamma^{(i)};j_1,j_2} (\bfm{x}) = 
\begin{cases}
  (\g_{\Gamma^{(i)}; j_1,j_2}^{(i_0)} \circ (\ab{F}^{(i_0)})^{-1})(\bfm{x}) \;
\mbox{ if }\f \, \bfm{x} \in \overline{\Omega^{(i_0)}},
\\[0.15cm] 
  (\g_{\Gamma^{(i)}; j_1,j_2}^{(i_1)} \circ (\ab{F}^{(i_1)})^{-1} )(\bfm{x}) \;
\mbox{ if }\f \, \bfm{x} \in \overline{\Omega^{(i_1)}}, 
\\[0.15cm] 
  0\quad {\rm otherwise},
\end{cases}
\end{equation}
and where the functions $\g_{\Gamma^{(i)}; j_1,j_2}^{(\tau)}$, $\tau \in \{i_0,i_1\}$, 
are specified in \eqref{eq:basisFunctionsGenericG}.
$\mathcal{W}^\sm_{\Gamma^{(i)}} \subseteq  \mathcal{V}^\sm$ follows from the fact that the functions $\phi_{\Gamma^{(i)};j_1,j_2}$, 
$j_2=2\sm+1 -j_1,\ldots, n_{j_1}+j_1-(2\sm+2)$, $j_1=0,1,\ldots,\sm,$ possess a support contained in $\overline{\Omega^{(i_0)}} \cup 
\overline{\Omega^{(i_1)}}$, are $C^s$-smooth at the edge~$\Gamma^{(i)}$ by construction, and have vanishing values and derivatives up to 
order $\sm$ at all other edges~$\overline{\Gamma^{(\ell)}}$, $\ell \in \mathcal{I}_\Gamma \setminus \{ i \}$.

\subsection{The vertex subspaces }

We will denote by $v_i$ the patch valency of a vertex~$\Xi^{(i)}$, $i \in \mathcal{I}_{\Xi}$. To generate the vertex 
subspaces~$\mathcal{W}^\sm_{\Xi^{(i)}}$, we will distinguish between different types of vertices~$\Xi^{(i)}$, namely between 
inner and boundary vertices, and in the latter case also between boundary vertices of patch valency~$v_{i} \geq 3$, $v_{i} =2$ and 
$v_{i}=1$. We will follow a similar approach
as used in \cite{KaSaTa19b,KaSaTa19a} and \cite{KaVi19a,KaVi20} for the construction of $C^1$ and $C^2$-smooth isogeometric spline
functions in the vicinity of the vertex~$\Xi^{(i)}$.

Let us start with the case of an inner vertex~$\Xi^{(i)}$, $i \in \mathcal{I}_{\Xi}$.
We can assume without loss of generality that all patches $\Omega^{(i_\rho)}$, $\rho=0,1,\ldots,v_i-1$, around the vertex~$\Xi^{(i)}$, 
i.e. $\Xi^{(i)} = \cap_{\rho=0}^{v_i-1} \overline{\Omega^{(i_\rho)}}$, are parameterized and labeled as shown in Fig.~\ref{fig:MultiPatchDomain}.  
In addition, we relabel the common edges $\overline{\Omega^{(i_\rho)}} \cap \overline{\Omega^{(i_{\rho+1})}}$, $\rho=0,1,\ldots,v_i-1$, 
by $\Gamma^{(i_{\rho+1})}$, where we take the lower index~$\rho$ of the indices~$i_{\rho}$ modulo~$v_i$.

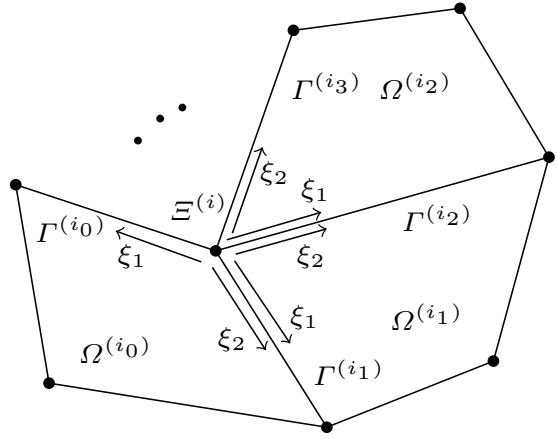
\begin{figure}[hbt]
\begin{center}
\resizebox{0.5\textwidth}{!}{
 \begin{tikzpicture}
  \coordinate(A) at (0,0);  \coordinate(B) at (3.,0.85); \coordinate(C) at (2.2,2.2); \coordinate(D) at (-1.8,0.6);
  \coordinate(E) at (-1.5,-1.2); \coordinate(F) at (1.0,-1.6); \coordinate(G) at (2.5,-1); 
  \coordinate(H) at (0.7,2);
  \coordinate(X) at (-0.3,1.3);
  \coordinate(Y) at (-0.7,1);
  \coordinate(Z) at (-0.5,1.2);

  \draw (A) -- (B); \draw (A) -- (D);  \draw (A) -- (F); 
  \draw (B) -- (C);  \draw (D) -- (E); \draw (E) -- (F); \draw (F) -- (G); \draw (G) -- (B);
  \draw (C) -- (H); \draw (H) -- (A);

  \fill (A) circle (1.5pt); \fill (B) circle (1.5pt); \fill (C) circle (1.5pt); \fill (D) circle (1.5pt);
  \fill (E) circle (1.5pt); \fill (F) circle (1.5pt); \fill (G) circle (1.5pt);  \fill (H) circle (1.5pt); 
  \fill (X) circle (1pt);
  \fill (Y) circle (1pt);
  \fill (Z) circle (1pt);

  \draw[->] (0.18,-0.03) -- (1.,0.2);
  \draw[->] (0.1,0.1) -- (0.95,0.35);
  \draw[->] (0.17,-0.09) -- (0.67,-0.84);
  \draw[->] (-0.03,-0.15) -- (0.45,-0.9);
  \draw[->] (-0.13,-0.1) -- (-0.9,0.18);
  \draw[->] (0.15,0.16) -- (0.42,0.94);

  \node at (0.86,-0.05) {\scriptsize $\xi_2$};
  \node at (0.9,0.55) {\scriptsize $\xi_1$};
  \node at (0.53,0.7) {\scriptsize $\xi_2$};
  \node at (-0.75,-0.05) {\scriptsize $\xi_1$};
  \node at (0.15,-0.8) {\scriptsize $\xi_2$};
  \node at (0.8,-0.6) {\scriptsize $\xi_1$};
  \node at (-0.9,-0.9) {\scriptsize $\Omega^{(i_0)}$};
  \node at (1.9,-0.6) {\scriptsize $\Omega^{(i_1)}$};
  \node at (1.8,1.5) {\scriptsize $\Omega^{(i_2)}$};
  \node at (-0.15,0.4) {\scriptsize $\Xi^{(i)}$};
  \node at (-1.3,0.2) {\scriptsize {$\Gamma^{(i_0)}$}};
  \node at (1.2,-1.1) {\scriptsize {$\Gamma^{(i_1)}$}};
  \node at (2.0,0.3) {\scriptsize {$\Gamma^{(i_2)}$}};
  \node at (1.0,1.5) {\scriptsize {$\Gamma^{(i_3)}$}};
  \end{tikzpicture}
 }
\end{center}
\caption{The parameterization of the patches $\Omega^{(i_0)},\Omega^{(i_1)},\ldots,\Omega^{(i_{v_i-1})}$ with the 
edges~$\Gamma^{(i_0)},\Gamma^{(i_1)}, \ldots,  \Gamma^{(i_{v_i-1})} $ around the inner vertex $\Xi^{(i)}$.}
\label{fig:MultiPatchDomain}
\end{figure}

The design of the subspace~$\mathcal{W}^\sm_{\Xi^{(i)}}$ is based on the construction of $C^s$-smooth functions in the vicinity of 
the vertex~$\Xi^{(i)}$, which will be formed by the linear combination of functions~$\phi_{\Gamma^{(i_{\rho})};j_1,j_2}$, $j_2 = 0,1,\ldots, 
2\sm -j_1,\; j_1=0,1,\ldots,\sm$, $\rho = 0,1,\ldots, v_i-1$, coinciding at their common supports in the vicinity of the vertex~$\Xi^{(i)}$, 
and by subtracting those ``standard'' isogeometric spline functions~$\phi_{\Omega^{(i_{\rho})}; j_1, j_2}$, $j_1,j_2=0,1,\ldots,s$, 
$\rho=0,1,\ldots,v_i-1$, which have been added twice. For this purpose, let us consider the isogeometric spline function
\begin{equation}  \label{eq:defphiXi}
  \phi_{\Xi^{(i)}} (\ab{x}) = 
  \begin{cases}
   \left( \ff_{i_\rho} \circ (\ab{F}^{(i_\rho)})^{-1}\right)(\bfm{x}) \;
\mbox{ if }\f \, \bfm{x} \in \overline{\Omega^{(i_\rho)}},\; \rho=0,1,\ldots,{v}_i -1,
\\
0 \quad \mbox{ otherwise,}
\end{cases}
\end{equation}
where the functions~$ \ff_{i_{\rho}}$ are given as 
\begin{align}  \label{eq:g_vertex}
 \ff_{i_{\rho}}(\xi_1,\xi_2)    =   f_{i_\rho}^{\Gamma^{(i_\rho)}}(\xi_1,\xi_2)
  +  f_{i_\rho}^{\Gamma^{(i_{\rho+1})}}(\xi_1,\xi_2) \, - 
 f_{i_\rho}^{\Omega^{(i_\rho)}}(\xi_1,\xi_2), 
  \end{align}
  with
\begin{align} \label{eq:g_vertex2} 
    f_{i_\rho}^{\Gamma^{(i_{\rho+\tau})}}(\xi_1,\xi_2)  & = \sum_{j_1=0}^\sm \sum_{j_2=0}^{2 \sm -j_1} a^{\Gamma^{(i_{\rho+\tau})}}_{j_1,j_2} \, 
  f_{\Gamma^{(i_{\rho+\tau}); j_1,j_2}}^{(i_\rho)} (\xi_{2-\tau},\xi_{1+\tau}), \quad { a^{\Gamma^{(i_{\rho+\tau})}}_{j_1,j_2} \in \R},
  \quad \tau = 0,1, \nonumber \\[-0.3cm] 
  \\[-0.3cm]
 f_{i_\rho}^{\Omega^{(i_\rho)}}(\xi_1,\xi_2)  & =  \sum_{j_1=0}^{\sm} \sum_{j_2=0}^{\sm } a^{(i_\rho)}_{j_1,j_2} N_{j_1,j_2}^{\ab{p},\ab{r}} 
 (\xi_1,\xi_2), \quad {a_{j_1,j_2}^{(i_\rho)} \in \R }, \nonumber
\end{align}
and with the functions~$f_{\Gamma^{(i_{\rho+\tau}); j_1,j_2}}^{(i_\rho)}$, $\tau=0,1,$ given in~\eqref{eq:basisFunctionsGenericG}.
The function~$\phi_{\Xi^{(i)}}$ is now $C^\sm$-smooth on~$\overline{\Omega}$, i.e.
{$\phi_{\Xi^{(i)}} \in \mathcal{V}^\sm$,}
if the coefficients {$a^{\Gamma^{(i_{\rho+\tau})}}_{j_1,j_2}$, $a_{j_1,j_2}^{(i_\rho)}$} satisfy the equations
\begin{equation} \label{eq:vertex_homogeneous_system}
  \partial_{1}^{\ell_1}  
  \partial_{2}^{\ell_2} 
   \left( f_{i_\rho}^{\Gamma^{(i_{\rho+1})}} - f_{i_\rho}^{\Gamma^{(i_\rho)}} \right)  (\bfm{0}) = 0 \quad \mbox{ and } \quad
   \partial_{1^{}}^{\ell_1}
   \partial_{2^{}}^{\ell_2}
   \left( f_{i_\rho}^{\Gamma^{(i_{\rho+1})}} - f_{i_\rho}^{\Omega^{(i_\rho)}} \right)  (\bfm{0}) = 0, 
\end{equation} 
for $ 0 \leq \ell_1, \ell_2 \leq \sm$ and $\rho =0,1, \ldots, v_{i}-1$.
The equations~\eqref{eq:vertex_homogeneous_system} form a homogeneous system of linear equations
\begin{equation} \label{eq:homogoneous_system_Large}
 T^{(i)} \bfm{a} = \f{0},
\end{equation}
where the vector~$\bfm{a}$ consists of all coefficients $a^{\Gamma^{(i_{\rho+\tau})}}_{j_1,j_2}$, $a_{j_1,j_2}^{(i_\rho)}$.
Any choice of the vector~$\bfm{a}$, which fulfills the linear system \eqref{eq:homogoneous_system_Large}, yields an isogeometric 
function~\eqref{eq:defphiXi} 
belonging to the spline space~$\mathcal{V}^\sm$.
Each basis of the kernel $\ker (T^{(i)})$ determines 
$\dim \ker (T^{(i)})$ linearly independent $C^\sm$-smooth isogeometric 
spline functions, which are denoted by $\phi_{\Xi^{(i)}, j}$, $j \in \{0,1,\ldots,
\dim \ker (T^{(i)})-1\}$, 
and which can be used to define the vertex subspace~$\mathcal{W}^\sm_{\Xi^{(i)}}$ via
\[
\mathcal{W}^\sm_{\Xi^{(i)}} = \Span \left\{ \phi_{{\Xi^{(i)}}, j} \;|\; j \in \{0,1,\ldots,
\dim \ker (T^{(i)})-1\} \right\} \subseteq \mathcal{V}^\sm.
\]
As in the case of the three-patch domain in the numerical examples in Section~\ref{sec:examples}, see 
Fig.~\ref{fig:examples} (bottom row), we can employ the algorithm developed in~\cite{KaVi17a} for computing a basis of $\ker (T^{(i)})$, which 
is based on the concept of minimal determining sets (cf.\cite{LaSch07}) for the coefficients~$\bfm{a}$. 

Let us continue with the three different mentioned cases of a boundary vertex~$\Xi^{(i)}$, $i \in \mathcal{I}_{\Xi}$. A boundary 
vertex~$\Xi^{(i)}$ of patch valency~$v_{i} \geq 3$ can be handled similarly as an inner vertex 
by assuming that the two boundary edges are labeled as $\Gamma^{(i_0)}$ and $\Gamma^{(i_{v_{i}})}$. 
Then, the only difference in the construction of the $C^s$-smooth functions~$\phi_{\Xi^{(i)}, j}$ and of the 
$C^s$-smooth space~$\mathcal{W}^{\sm}_{\Xi^{(i)}} \subseteq \mathcal{V}^{s}$ is that for the patches~$\Omega^{(i_0)}$ and $\Omega^{(i_{v_i-1})}$ 
the functions~$f_{i_0;j_1,j_2}^{\Gamma^{(i_0)}}$ and $f_{i_{v_{i}-1};j_1,j_2}^{\Gamma^{(i_{v_i})}}$ in~\eqref{eq:g_vertex2} are 
just the standard B-splines.

In case of a boundary vertex~$\Xi^{(i)}$ of patch valency~$v_i=2$, we can assume without loss of generality that the two 
neighboring patches~$\Omega^{(i_0)}$ and $\Omega^{(i_1)}$, $i_0,i_1 \in \mathcal{I}_{\Omega}$, which contain the 
vertex~$\Xi^{(i)}$ and possess the common edge~$\overline{\Gamma^{(j_0)}}= \overline{\Omega^{(i_0)}} \cap \overline{\Omega^{(i_1)}}, j_0 
\in \mathcal{I}_{\Gamma}$, are parameterized as shown in Fig.~\ref{fig:twopatchCase} and that the vertex~$\Xi^{(i)}$ is further given as 
$\Xi^{(i)} = \ab{F}^{(i_0)}(\ab{0}) = \ab{F}^{(i_1)}(\ab{0})$. Then, the vertex subspace~$\mathcal{W}^{\sm}_{\Xi^{(i)}}$ is generated as
\begin{align*} \label{eq:spaceWXB}
\mathcal{W}^\sm_{\Xi^{(i)}} =&  \Span \left\{ \widetilde{\phi}_{\Xi^{(i)}; j_1,j_2} |\;   
j_1=0,1,\ldots,3\sm, \; j_2 = \begin{cases}
                         0,1,\ldots, 2\sm-j_1 & \mbox{if }j_1 \leq 2\sm \\
                         0,1,\ldots, 3\sm-j_1 & \mbox{if }j_1 > 2\sm 
                                          \end{cases}
 \; \right \} ,
\end{align*}
with the functions
\begin{equation*} \label{eq:phi_vertex_2}
 \widetilde{\phi}_{\Xi^{(i)};j_1,j_2} (\bfm{x}) = 
    \begin{cases} 
     \phi_{\Gamma^{(j_0)};j_1,j_2}(\ab{x}) & \mbox{if }j_1=0,1,\ldots ,\sm\\
     \phi_{\Omega^{(i_{0})};j_1,j_2}(\ab{x}) & \mbox{if }j_1=\sm+1,\sm+2,\ldots,2\sm\\
     \phi_{\Omega^{(i_{1})};j_1-\sm,j_2}(\ab{x}) & \mbox{if }j_1=2\sm+1,2\sm+2,\ldots,3\sm,
    \end{cases}
\end{equation*}
where the functions $\phi_{\Omega^{(i_{\ell})};j_1,j_2}$, $\ell=0,1$, and $\phi_{\Gamma^{(j_0)}_{};j_1,j_2}$ are defined as 
in~\eqref{eq:PhiOmega2} and \eqref{eq:defphiGamma}, respectively. $\mathcal{W}^{\sm}_{\Xi^{(i)}} \subseteq \mathcal{V}^{\sm}$ results from the fact that 
the functions~$\phi_{\Gamma^{(j_0)};j_1,j_2}(\ab{x})$, $j_1=0,1,\ldots,\sm$, $j_2=0,1,\ldots,2\sm-j_1$, are $C^{\sm}$-smooth by construction 
on~$\overline{\Omega}$, and that the functions $\phi_{\Omega^{(i_{\ell})};j_1,j_2}$, $\ell=0,1$, $j_1=\sm+1,\sm+2,\ldots,3\sm$, $j_2=0,1,\ldots,2\sm-j_1$ if 
$j_1 \leq 2\sm$ and $j_2=0,1,\ldots,3\sm-j_1$ if $j_1 > 2\sm$, possess a 
support in $\overline{\Omega^{(i_{\ell})}}$, are $C^{\sm}$-smooth on $\overline{\Omega^{(i_\ell)}}$, and have vanishing values and derivatives
of order $\leq \sm$ along all inner edges~$\overline{\Gamma^{(j)}}$, $j \in \mathcal{I}_{\Gamma}$.

Finally, in case of a boundary vertex~$\Xi^{(i)}$ of patch valency $v_i=1$, we can assume without loss of generality that 
the boundary vertex $\Xi^{(i)}$ is given by $\Xi^{(i)}= \ab{F}^{(i_0)}(\ab{0})$, $i_0 \in \mathcal{I}_{\Omega}$. Then, 
the vertex subspace~$\mathcal{W}^{\sm}_{\Xi^{(i)}}$ is defined as
\begin{align*} 
\mathcal{W}^\sm_{\Xi^{(i)}} =&  \Span \left\{ \phi_{\Omega^{(i_0)}; j_1,j_2} |\;   
j_1,j_2 =  0,1,\ldots, 2 \sm, \; j_1+j_2 \leq 2 \sm \right\} ,
\end{align*} 
where the functions $\phi_{\Omega^{(i_0)}; j_1,j_2}$ are given as in~\eqref{eq:PhiOmega2}.
Again, the functions~$\phi_{\Omega^{(i_0)}; j_1,j_2}$, $j_1,j_2 =  0,1,\ldots, 2 \sm$, $j_1+j_2 \leq 2 \sm$, are entirely contained 
in~$\overline{\Omega^{(i_0)}}$, are $C^{\sm}$-smooth on $\overline{\Omega^{(i_0)}}$, and have vanishing values and derivatives of order 
$\leq \sm$ along all inner edges~$\overline{\Gamma^{(j)}}$, $j \in \mathcal{I}_{\Gamma}$, which implies that the functions are $C^\sm$-smooth 
on~$\overline{\Omega}$, and therefore $\mathcal{W}^{\sm}_{\Xi^{(i)}} \subseteq \mathcal{V}^{(\sm)}$.

Since the functions which have been used to generate the spaces~$\mathcal{W}^{\sm}_{\Omega^{(i)}}$, $i \in \mathcal{I}_{\Omega}$, 
$\mathcal{W}^{\sm}_{\Gamma^{(i)}}$, $i \in \mathcal{I}_{\Gamma}$, and $\mathcal{W}^{\sm}_{\Xi^{(i)}}$, $i \in \mathcal{I}_{\Xi}$, are 
linearly independent by definition and/or by construction, they form a basis of the individual spaces $\mathcal{W}^{\sm}_{\Omega^{(i)}}$, 
$\mathcal{W}^{\sm}_{\Gamma^{(i)}}$ and $\mathcal{W}^{\sm}_{\Xi^{(i)}}$, and therefore, they build a basis of the space~$\mathcal{W}^{\sm}$.
By the direct sum~\eqref{eq:W}, the dimension of $\mathcal{W}^\sm$ is equal to
$$
 \dim \mathcal{W}^\sm = \sum_{i \in \mathcal{I}_\Omega} \dim \mathcal{W}_{\Omega^{(i)}}^\sm  +  \sum_{i \in \mathcal{I}_\Gamma} 
 \dim \mathcal{W}_{\Gamma^{(i)}}^\sm  + \sum_{i \in \mathcal{I}_\Xi} \dim \mathcal{W}^\sm_{\Xi^{(i)}},
$$
where
$$
 \dim \mathcal{W}_{\Omega^{(i)}}^\sm = (n - 2(\sm+1))^2 , 
$$
$$
 \dim \mathcal{W}_{\Gamma^{(i)}}^\sm = \begin{cases}
   (\sm+1) \left(n-k \sm - \left( \frac{7 \sm}{2} + 2\right)\right)& \mbox{if $\Gamma^{(i)}$ is an inner edge}, \\
  (\sm+1) (n-3\sm-2) & \mbox{if $\Gamma^{(i)}$ is a boundary edge},
 \end{cases}
$$
and
$$
\dim \mathcal{W}^\sm_{\Xi^{(i)}} = \begin{cases}
                                    \dim \ker (T^{(i)}) & \mbox{if }\Xi^{(i)} \mbox{ is an inner or boundary vertex with }v_i\geq3,\\
                                     (\sm+1) (\frac{5}{2}\sm+1)& \mbox{if }\Xi^{(i)} \mbox{ is a boundary vertex with }v_i=2, \\
                                    (\sm+1)(2\sm+1) & \mbox{if }\Xi^{(i)} \mbox{ is a boundary vertex with }v_i=1.
                                   \end{cases}
$$
\begin{remark}
In case of a bilinearly parameterized two-patch domain, the two slightly different constructions described in this and in the previous section 
lead in both cases to the same subspace~$\mathcal{W}^{\sm}$ with the same basis, which can be easily verified by comparing the two differently 
generated bases. 
\end{remark}

Note that in case of an inner vertex or in case of a boundary vertex of 
patch valency~$v_{i} \geq 3$, $\dim \ker (T^{(i)})$, and hence $\dim \mathcal{W}^{\sm}_{\Xi^{(i)}}$, does not just depend
on the valency $v_i$ of the vertex~$\Xi^{(i)}$ but also
on the configuration of the bilinear patches around the corresponding vertex. 
An alternative approach for the case of an inner vertex and for the case of a boundary vertex of patch valency~$v_{i} \geq 3$, which leads also in 
these two cases to a vertex subspace whose dimension is independent of the configuration of the bilinear patches, 
is to enforce additionally $C^{2\sm}$-smoothness of the functions at the vertex~$\Xi^{(i)}$, see e.g.~\cite{KaSaTa19a,KaSaTa19b} and \cite{KaVi20} for 
$\sm=1$ and $\sm=2$, respectively.
Thereby, we will compute a subspace~$\widehat{\mathcal{W}}^{\sm}_{\Xi^{(i)}}$ of the vertex subspace~$\mathcal{W}^{\sm}_{\Xi^{(i)}}$, i.e.,
$\widehat{\mathcal{W}}^{\sm}_{\Xi^{(i)}} \subseteq \mathcal{W}^{\sm}_{\Xi^{(i)}}$.
For this purpose, let $\psi_{j_1,j_2}: \overline{\Omega} \to \R$, $j_1,j_2=0,1,\ldots,2\sm$ with $j_1+j_2 \leq 2\sm $, be functions which are 
$C^\sm$-smooth on $\overline{\Omega}$ and additionally $C^{2\sm}$-smooth at the 
vertex~$\Xi^{(i)}$ such that
\[
 \partial_{1}^{\ell_1} \partial_{2}^{\ell_2} \psi_{j_1,j_2} \left(\Xi^{(i)}\right) = 
 \sigma^{\ell_1+\ell_2} \,\delta^{\ell_1}_{j_1} \delta^{\ell_2}_{j_2}, \quad
 \ell_1,\ell_2=0,1,\ldots,2\sm, \; \ell_1+\ell_2 \leq 2\sm ,
\]
where 
$\sigma$ is a scaling factor (cf.~\cite{KaSaTa19a}) given by
\[
 \sigma = \left(\frac{h}{p \, v_i} \sum_{\rho=0}^{v_i -1} ||{ {\rm J} \ab{F}^{(i_{\rho})}(\ab{0})} || \right)^{-1},
\] 
with
${\rm J} \ab{F}^{(i_{\rho})}$ being the Jacobian of $\ab{F}^{(i_{\rho})}$.
Then, isogeometric functions~$\widehat{\phi}_{\Xi^{(i)};j_1,j_2}$, $j_1,j_2=0,1,\ldots,2\sm$, $j_1+j_2 \leq 2 \sm$, 
can be defined via $\widehat{\phi}_{\Xi^{(i)};j_1,j_2}=\phi_{\Xi^{(i)}}$, 
with the functions $\phi_{\Xi^{(i)}}$ given in~\eqref{eq:defphiXi}, by means of the interpolation problem
\begin{equation} \label{eq:interpolation_conditions}
 \partial^{\ell_1}_{1} \partial^{\ell_2}_{2} \phi_{\Xi^{(i)}} \left(\Xi^{(i)}\right) = \partial_{1}^{\ell_1} \partial_{2}^{\ell_2} \psi_{j_1,j_2} 
 \left(\Xi^{(i)}\right), \quad
 \ell_1,\ell_2=0,1,\ldots,2\sm, \mbox{ } \ell_1+\ell_2 \leq 2 \sm.
\end{equation}
The isogeometric functions~$\phi_{\Xi^{(i)}}$, and therefore the isogeometric functions $\widehat{\phi}_{\Xi^{(i)};j_1,j_2}$, are uniquely determined 
by~\eqref{eq:interpolation_conditions} and can be computed via the coefficients~$a^{\Gamma^{(i_{\rho+\tau})}}_{j_1,j_2}$ and 
$a_{j_1,j_2}^{(i_\rho)}$ of the spline functions~$\ff_{i_{\rho}}$ in \eqref{eq:g_vertex} with the help of the following equivalent interpolation 
conditions
\begin{align*} 
   \partial_{1}^{\ell_1}  \partial_{2}^{\ell_2} \, f_{i_\rho}^{\Gamma^{(i_\rho)}}  (\bfm{0}) &=
     \partial_{1}^{\ell_1}   \partial_{2}^{\ell_2} \left( \psi_{j_1,j_2} \circ \bfm{F}^{(i_\rho)} \right)  (\bfm{0}),\ \;
   0 \leq \ell_1 \leq 2\sm, \;  0 \leq \ell_2 \leq \sm, \;  \ell_1+\ell_2 \leq 2\sm, \nonumber \\[-0.3cm] 
 & \label{eq:systemEq} \\[-0.3cm]   
        \partial_{1}^{\ell_1}  \partial_{2}^{\ell_2} \, f_{i_\rho}^{\Omega^{(i_{\rho})}} (\bfm{0}) &=
     \partial_{1}^{\ell_1}   \partial_{2}^{\ell_2} \left( \psi_{j_1,j_2} \circ \bfm{F}^{(i_\rho)} \right) (\bfm{0}), \ \;
   0 \leq \ell_1,\ell_2 \leq \sm, \nonumber
\end{align*}
for $\rho = 0,1,\ldots, v_{i}-1$. The resulting isogeometric spline functions~$\widehat{\phi}_{\Xi^{(i)};j_1,j_2}$, $j_1,j_2=0,1,\ldots,2\sm$, 
$j_1+j_2 \leq 2\sm$ are well-defined, $C^{\sm}$-smooth on~$\overline{\Omega}$ and even $C^{2\sm}$-continuous at the vertex~$\Xi^{(i)}$, and determine 
by
$$
 \widehat{\mathcal{W}}^\sm_{\Xi^{(i)}}=  \Span \left\{ \widehat{\phi}_{\Xi^{(i)};j_1,j_2} |\; \; j_1,j_2=0,1,\ldots,2\sm, \; j_1+j_2 \leq 2\sm \right\} 
 \subseteq \mathcal{W}^{\sm}_{\Xi^{(i)}} \subseteq  \mathcal{V}^{\sm},
$$
a vertex subspace~$\widehat{\mathcal{W}}^{\sm}_{\Xi^{(i)}}$ with a dimension which is independent of the valency $v_i$ of the vertex~$\Xi^{(i)}$ and  
of the configuration of the bilinear patches around the vertex, since it just equals
\[
\dim \widehat{\mathcal{W}}^\sm_{\Xi^{(i)}} = (s+1)(2s+1).
\]

\section{Beyond bilinear parameterization} \label{sec:bilinearlike}

In this section, we will briefly discuss a first possible generalization of the presented construction to a wider class of multi-patch 
parameterizations than the considered bilinear one. Motivated by~\cite{CoSaTa16} for $\sm=1$ and by~\cite{KaVi17c} for $\sm=2$, we are interested in multi-patch parameterizations which possess 
similar connectivity functions as in the bilinear case, in particular linear functions~$\alpha^{(\LL)}$, $\alpha^{(\RR)}$, $\beta^{(\LL)}, 
\beta^{(\RR)}$, along the interfaces~$\Gamma^{(i)}$, $i \in \mathcal{I}_{\Gamma}$. There, but also in further publications for the case~$\sm=1$ or 
$\sm=2$, see e.g.~\cite{KaSaTa19a,KaSaTa19b,KaVi19a}, it was numerically shown that such multi-patch parameterizations can allow similar to the bilinear 
case the construction of globally $C^s$-smooth isogeometric spline spaces with optimal approximation properties. Inspired by~\cite{KaVi17c}, we 
call these particular multi-patch parameterizations bilinear-like $G^s$ and define them as follows.

\begin{definition}  \label{def:bilinearlike}
A multi-patch parameterization~$\ab{F}$ consisting of the geometry mappings~
$$
\ab{F}^{(i)}\in 
\mathcal{S}^{\ab{p},\ab{r}}_{h}([0,1]^{2}) \times 
\mathcal{S}^{\ab{p},\ab{r}}_{h}([0,1]^{2}), \quad i \in \mathcal{I}_\Omega,
$$
is called \emph{bilinear-like $G^{\sm}$} if for any two neighboring 
patches~$\ab{F}^{(i_0)}$ and $\ab{F}^{(i_1)}$, $i_0,i_1 \in \mathcal{I}_{\Omega}$, assuming without loss of generality that
\begin{equation*} \label{eq:FfuncG0} 
 \ab{F}^{(\LL)}(0,\xi_2) = \ab{F}^{(\RR)}(0,\xi_2) ,
\end{equation*}
there exist linear functions 
$\alpha^{(\LL)},\alpha^{(\RR)}$, $\beta^{(\LL)}$, 
$\beta^{(\RR)} 
$, such that
 \begin{equation*}   \label{eq:FC}
 \ab{F}_\ell^{(\LL)}(\xi) = \ab{F}_\ell^{(\RR)}(\xi) =: \ab{F}_\ell(\xi), \quad \ell =0,1,\ldots,\sm,
 \end{equation*}
 with
 \begin{equation*}   \label{eq:FC2}
 \ab{F}_\ell^{(\Side)}(\xi) = \left(\alpha^{(\Side)}(\xi)\right)^{-\ell}\, \partial_1^\ell \ab{F}^{(\Side)}(0,\xi) - \sum_{i=0}^{\ell-1} {\ell \choose i} 
 \left(\frac{\beta^{(\Side)}(\xi)}{\alpha^{(\Side)}(\xi)}\right)^{\ell-i}  \dd^{\ell-i} \ab{F}_i^{}(\xi) ,\quad \Side\in \{\LL,\RR\}.
 \end{equation*}
\end{definition}

For example, Theorem~\ref{thm:g_ell} and \ref{lem:explicit_second} can be directly applied by employing bilinear-like $G^s$ multi-patch parameterizations.
The advantage of using bilinear-like $G^\sm$ multi-patch parameterizations instead of bilinear multi-patch parameterizations is the possibility to 
deal with multi-patch domains with curved interfaces and curved boundaries, see e.g.~\cite{KaSaTa19a} and \cite{KaVi17c} for $\sm=1$ and for $\sm=2$, 
respectively. However, a detailed study about the generalization to bilinear-like $G^s$ multi-patch parameterizations is beyond the scope of the paper 
and will be a part of our planned future research. 

\section{Examples} \label{sec:examples}

\begin{figure}
\centering
\begin{tabular}{ccc}
\includegraphics[width=4.8cm,clip]{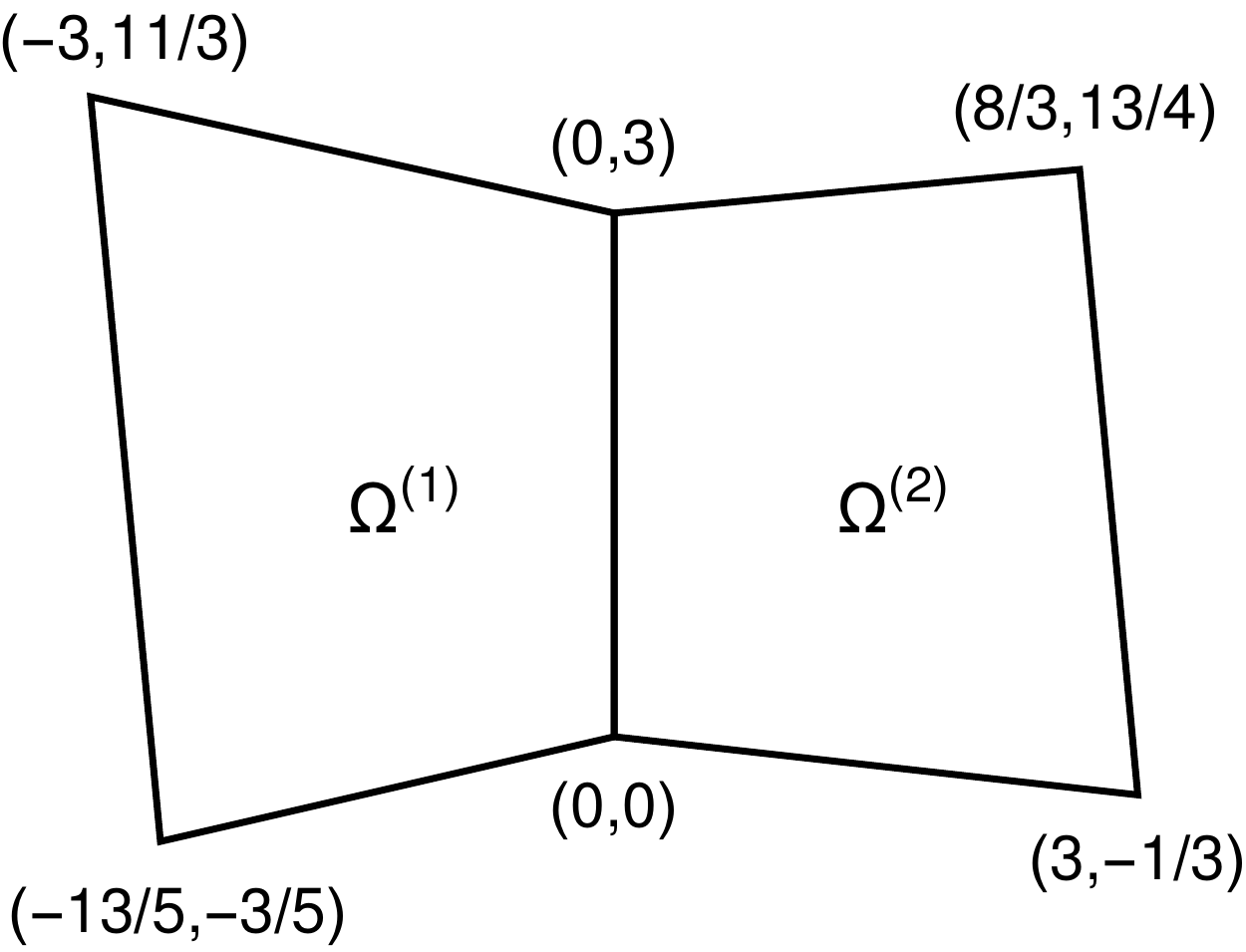} &
\includegraphics[width=4.1cm,clip]{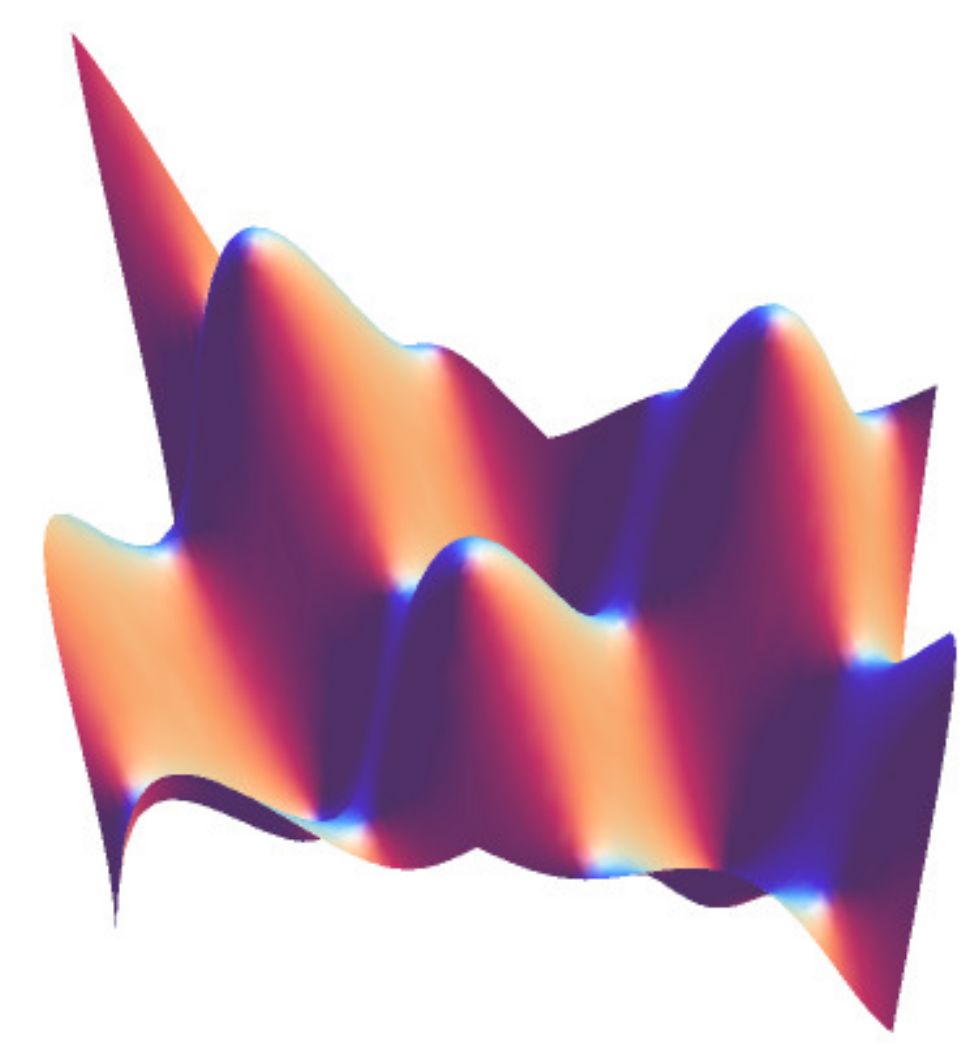} &
\includegraphics[width=4.7cm,clip]{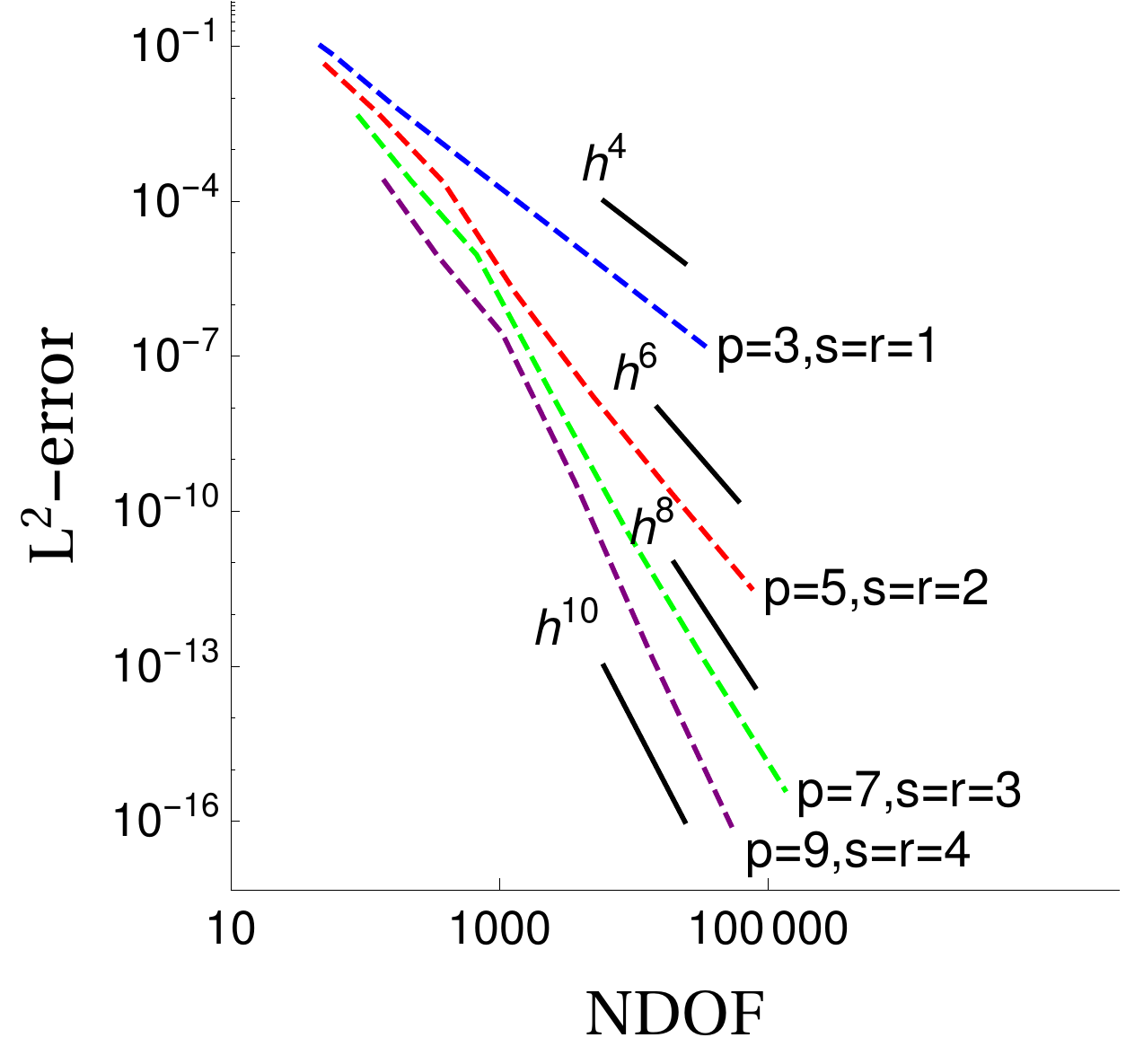} \\
\includegraphics[width=4.8cm,clip]{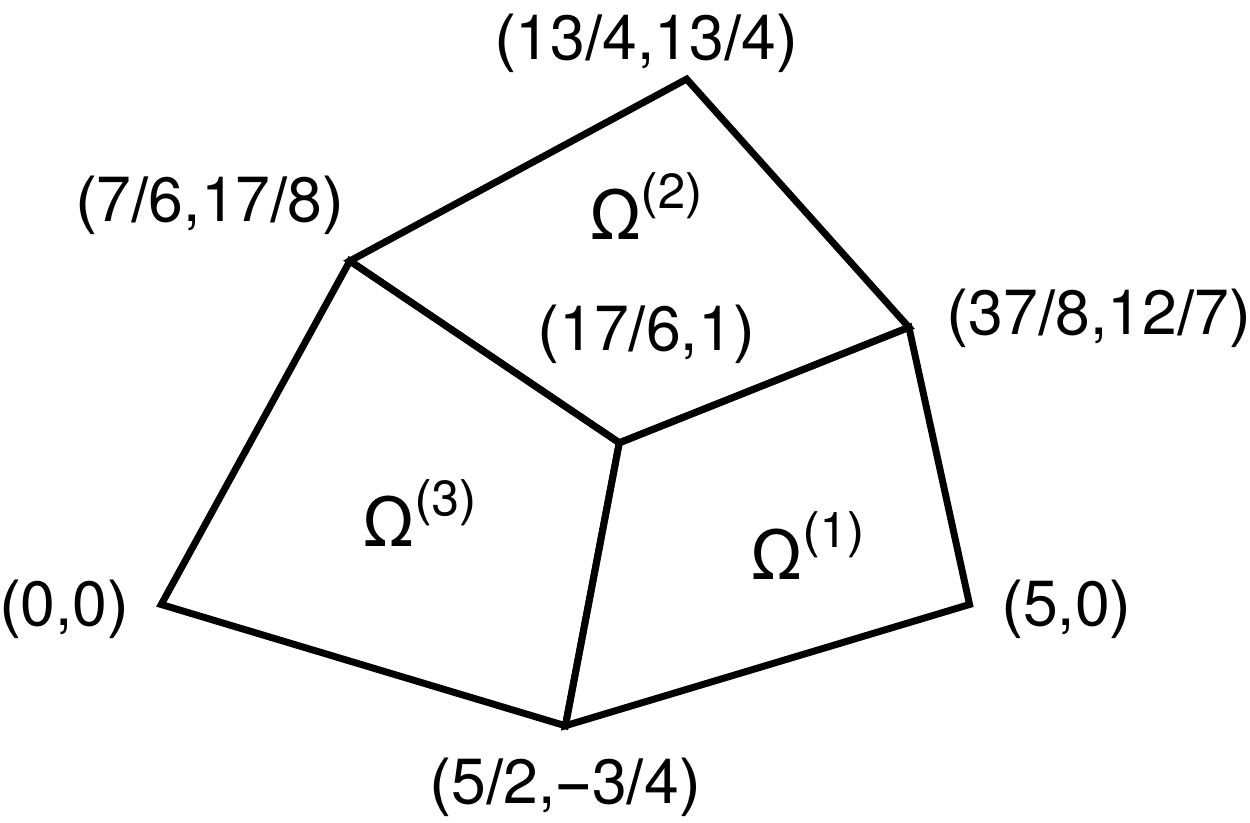} &
\includegraphics[width=4.1cm,clip]{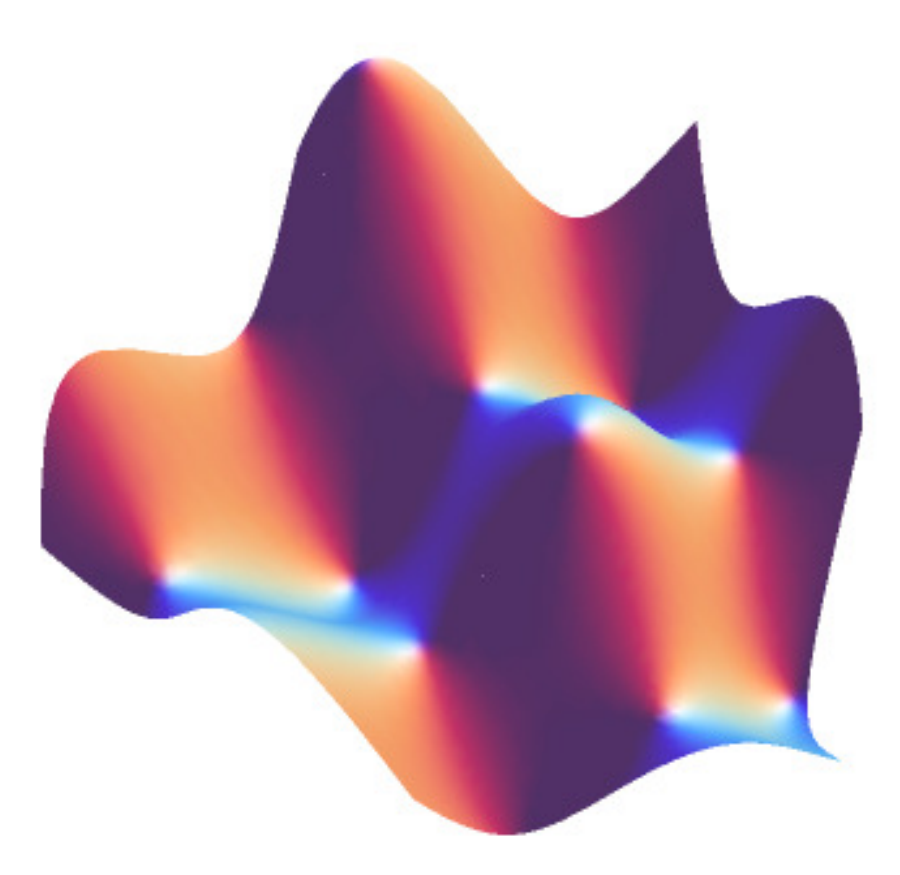} &
\includegraphics[width=4.7cm,clip]{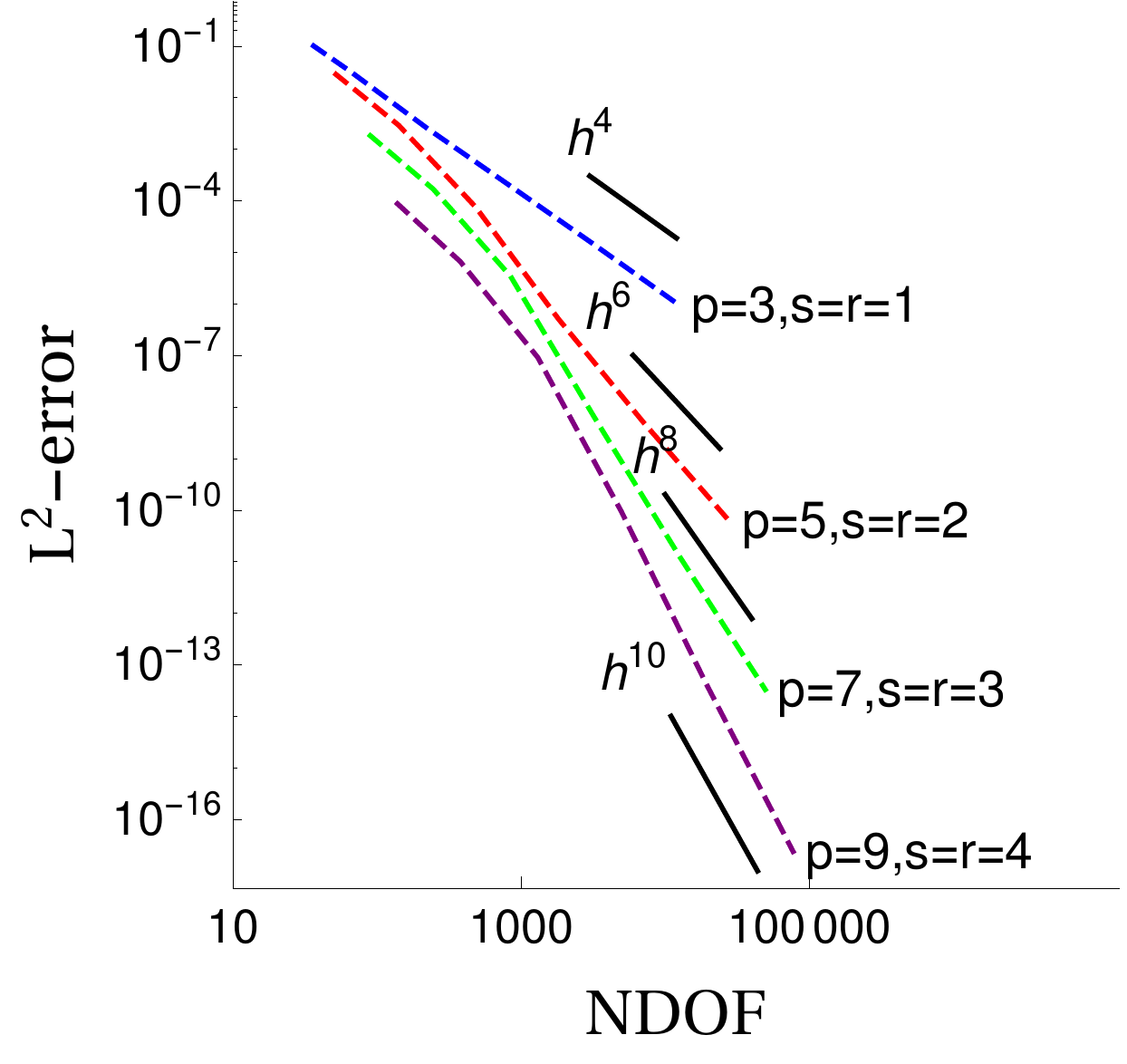} \\
Multi-patch domain~$\Omega$ & Smooth function~$z$ & Relative $L^2$ errors
\end{tabular}
\caption{$L^2$ approximation of the smooth function~\eqref{eq:smooth_function} (middle column) on two different bilinearly 
parameterized multi-patch 
domains~$\Omega$ (left column) using spline spaces~$\mathcal{W}^{\sm}$ for a global smoothness~$\sm=1,\ldots,4$, and a mesh 
size~$h=\frac{1}{2^L}$, with $L=0,1,\ldots,5$ or $L=0,1,\ldots,6$, for a spline degree~$p=2\sm +1$ and for an inner patch regularity~$r=\sm$. 
The resulting relative $L^2$ errors (right column) are visualized with respect to the number of degrees of freedom (NDOF).}
\label{fig:examples}
\end{figure}

The goal of this section is to numerically study the approximation power of the isogeometric spline space~$\mathcal{W}^{\sm}$ by performing 
$L^2$ approximation over the two bilinearly parameterized multi-patch domains~$\Omega$ given in Fig.~\ref{fig:examples} (left column). 
More precisely, we will approximate the smooth function 
\begin{equation} \label{eq:smooth_function}
z(\ab{x}) = z(x_1,x_2)= 4 \cos (2 x_1) \sin (2 x_2),
\end{equation}
visualized in Fig.~\ref{fig:examples}~(middle column) on these two multi-patch domains, by employing isogeometric spline 
spaces~$\mathcal{W}^{\sm}$ for a global 
smoothness~$\sm=1,\ldots,4$, and a mesh size~$h=\frac{1}{2^L}$, with $L=0,1,\ldots,5$ or $L=0,1,\ldots,6$, for a spline degree~$p=2\sm +1$ and for 
an inner patch regularity~$r=\sm$. Let $\{\phi_{j}\}_{j=0}^{\dim \mathcal{W}^{\sm}-1}$ be a basis of such an isogeometric spline 
space~$\mathcal{W}^{\sm}$, then we compute an approximation 
\[
 z_h(\ab{x}) = \sum_{j=0}^{\dim \mathcal{W}^{\sm}-1 } c_{j} \phi_{j}(\ab{x}), \quad c_{j} \in \R,
\]
of the function~$z$, by minimizing the objective function
\[
 \int_{\Omega} (z_h(\ab{x})-z(\ab{x}))^2 \mathrm{d}\ab{x}.
\]
Finding a solution of this minimization problem is equivalent to solving the linear system 
\[
 M \ab{c} = \ab{z}, \quad \ab{c}=(c_{j})_{j=0}^{\dim \mathcal{W}^{\sm} -1},
\]
where $M$ is the mass matrix with the single entries
\begin{equation} \label{eq:mass_entries}
 m_{j_1,j_2}= \int_{\Omega} \phi_{j_1} (\ab{x}) \phi_{j_2}(\ab{x}) \mathrm{d}\ab{x},
\end{equation}
and $\ab{z}$ is the right side vector with the single entries
\begin{equation} \label{eq:right_entries}
 z_{j} = \int_{\Omega} z(\ab{x}) \phi_{j} (\ab{x}) \mathrm{d}\ab{x}.
\end{equation}
Using the relation~$f_{j}^{(i)} = \phi_{j} \circ \ab{F}^{(i)}$, $i \in \mathcal{I}_{\Omega}$, the entries~\eqref{eq:mass_entries} and 
\eqref{eq:right_entries} can be computed via
\[
 m_{j_1,j_2}= \sum_{i \in \mathcal{I}_{\Omega}} \int_{[0,1]^2}  f_{j_1}^{(i)}(\xi_1,\xi_2)  f_{j_2}^{(i)}(\xi_1,\xi_2) 
 |\det ( {\rm J}\ab{F}^{(i)}(\xi_1,\xi_2))| \mathrm{d}\xi_1 \mathrm{d}\xi_2,
\]
and
\[
 z_{j} =  \sum_{i \in \mathcal{I}_{\Omega}} \int_{[0,1]^2} z(\ab{F}^{(i)}(\xi_1,\xi_2)) f_{j}^{(i)}(\xi_1,\xi_2) 
 |\det ( {\rm J} \ab{F}^{(i)}(\xi_1,\xi_2))| \mathrm{d}\xi_1 \mathrm{d}\xi_2,
\]
respectively.

While in case of the bilinearly parameterized two-patch domain, see Fig.~\ref{fig:examples} (top row and left column), the basis for the 
space~$\mathcal{W}^{\sm}$ is generated as described in Section~\ref{subsec:design_twopatches}, in case of the bilinearly parameterized three-patch 
domain, see Fig.~\ref{fig:examples} (bottom row and left column), the basis is constructed as explained in 
Section~\ref{sec:design_multipatches}. In the latter case, the design of the vertex subspaces~$\mathcal{W}^\sm_{\Xi^{(i)}}$ has to be 
slightly modified in case of a mesh size~$\frac{p-r-s}{3s-r+1} \leq h \leq 1$ for boundary vertices. Namely, the vertex 
subspace~$\mathcal{W}^\sm_{\Xi^{(i)}}$ for a boundary vertex~$\Xi^{(i)}$ is then just generated by those corresponding functions 
$\phi_{\Gamma^{(i)};j_1,j_2}$ and/or $\phi_{\Omega^{(i)};j_1,j_2}$, which have not been already used to construct the vertex subspace for another 
vertex especially for the inner vertex. 

Fig.~\ref{fig:examples} (right column) displays the resulting relative $L^2$ errors with respect to the number of degrees of freedom (NDOF) by 
performing $L^2$ approximation on the two different bilinearly parameterized multi-patch domains. In all cases, the numerical results 
indicate a convergence rate of optimal order of~$\mathcal{O}(h^{p+1})$ in the $L^2$ norm. In case of the three-patch domain, the shown results have 
been obtained by employing the minimal determing set approach for the construction of the vertex subspaces. However, the use of the alternative 
interpolation strategy instead would lead to a nearly indistinguishable result but which is not presented here. The number of degrees of freedom, 
i.e. the dimensions of the obtained isogeometric spline spaces~$\mathcal{W}^{\sm}$ for the two different multi-patch domains are reported in 
Table~\ref{tab:examples}. 

\begin{table}[htb]
 \centering\footnotesize 
 \begin{tabular}{|c|c|c|c|c|} \hline
 & \multicolumn{4}{|c|}{Two-patch domain} \\ \hline
 $h$ & $p=3$, $s=r=1$ & $p=5$, $s=r=2$ & $p=7$, $s=r=3$ & $p=9$, $s=r=4$ \\ \hline 
$1$            & 23   & 51    & 90    & 140 \\ \hline 
$\frac{1}{2}$  & 57   & 126   & 222   & 345 \\ \hline
$\frac{1}{4}$  & 173  & 384   & 678   & 1055 \\ \hline
$\frac{1}{8}$  & 597  & 1332  & 2358  & 3675 \\ \hline
$\frac{1}{16}$ & 2213 & 4956  & 8790  & 13715 \\ \hline
$\frac{1}{32}$ & 8517 & 19116 & 33942 & 52995 \\ \hline
\hline
  & \multicolumn{4}{|c|}{Three-patch domain} \\ \hline
 $h$ & $p=3$, $s=r=1$ & $p=5$, $s=r=2$ & $p=7$, $s=r=3$ & $p=9$, $s=r=4$ \\ \hline 
$1$            & 24 (24)       & 52 (51)       & 90 (88)       & 139 (135) \\ \hline 
$\frac{1}{2}$  & 66 (66)       & 142 (141)     & 246 (244)     & 379 (375) \\ \hline
$\frac{1}{4}$  & 222 (222)     & 484 (483)     & 846 (844)     & 1309 (1305) \\ \hline
$\frac{1}{8}$  & 822 (822)     & 1816 (1815)   & 3198 (3196)   & 4969 (4965) \\ \hline
$\frac{1}{16}$ & 3174 (3174)   & 7072 (7071)   & 12510 (12508) & 19489 (19485) \\ \hline
$\frac{1}{32}$ & 12486 (12486) & 27952 (27951) & 49566 (49564) & 77329 (77325) \\ \hline
  \end{tabular}
  \caption{The number of degrees of freedom, i.e. the dimensions of the generated isogeometric spline spaces~$\mathcal{W}^{\sm}$ in 
  Section~\ref{sec:examples} for a mesh size~$h=\frac{1}{2^L}$, $L=0,1,\ldots,5$, for the two bilinearly parameterized multi-patch 
domains shown 
  in Fig.~\ref{fig:examples} (left column). In case of 
  the three-patch domain, the number in the brackets represents the dimension when using the alternative interpolation strategy for the 
  construction of the vertex subspaces instead of the minimal determining set approach.}
  \label{tab:examples}
\end{table}

\section{Conclusion} \label{sec:conclusion}

We have studied the space of $C^{\sm}$-smooth ($\sm \geq 1$) isogeometric spline functions on planar, bilinearly parameterized multi-patch 
domains and have presented the construction of  a particular subspace of this $C^{\sm}$-smooth isogeometric spline space. The use of the 
$C^{\sm}$-smooth subspace is advantageous compared to the use of the entire $C^{\sm}$-smooth space, since the design of the subspace
is simple and works uniformly for all possible multi-patch configurations, and furthermore, the numerical experiments by performing 
$L^2$ approximation indicate that the subspace already possesses optimal approximation properties.

The construction of the $C^{\sm}$-smooth subspace and of an associated simple and locally supported basis is firstly described 
for the case of two-patch domains, and is then extended to the case of multi-patch domains with more than two patches and with possibly 
extraordinary vertices. In the latter case, the $C^{\sm}$-smooth subspace is generated as the direct sum of spaces corresponding to the 
individual patches, edges and vertices. 

Moreover, a possible generalization of our approach to a more general class of planar multi-patch parameterizations, called bilinear-like 
$G^{\sm}$ multi-patch geometries, is briefly explained. This class of multi-patch parameterizations provides the possibility to model multi-patch 
domains with curved interfaces and boundaries. A detailed study of this class of geometries is beyond the scope of the paper and is a topic of 
our future research. Further open problems which are worth to study are e.g. the theoretical investigation of the approximation 
properties of the constructed $C^{\sm}$-smooth isogeometric spline space, the use of the $C^{\sm}$-smooth isogeometric spline functions for 
applications which require functions of high continuity such as solving fourth order PDEs via isogeometric collocation, and the extension of our 
approch to multi-patch shells and volumes.

\begin{acknowledgements}
M. Kapl has been partially supported by the Austrian Science Fund (FWF) through the project P~33023.
V.~Vitrih has been partially supported by the research program P1-0404 and research projects J1-9186, J1-1715
 from ARRS, Republic of Slovenia.
This support is gratefully acknowledged. 
\end{acknowledgements}

%
 \section*{Conflict of interest}
 The authors declare that they have no conflict of interest.

\bibliographystyle{spmpsci}      

\end{document}